\documentclass[11pt,english]{article}
\usepackage{a4wide}
\usepackage{amssymb}
\usepackage{amsthm}
\usepackage{amsmath}
\usepackage{amsfonts}
\usepackage{latexsym}
\usepackage{amsxtra}
\usepackage{graphicx}
\usepackage{wrapfig}
\usepackage{mathrsfs}
\usepackage{times}
\usepackage{makeidx}
\usepackage{esint}
\usepackage{wasysym}
\usepackage{xr}
\usepackage{cite}
\usepackage[breaklinks=true,bookmarks=true,linktocpage=true,pdfpagelabels=true,plainpages=false]{hyperref}

\newcommand{\xx}{\mbox{$\clubsuit$}}
\newcommand{\yy}{\mbox{$\spadesuit$}}
\newcommand{\zz}{\mbox{$\bigstar$}}

\providecommand{\MR}[1]{}

\newcommand{\heikodetail}[1]{}

\newcommand{\reallyinvisible}[2]{}      

\newfont{\thickmath}{msbm10 scaled \magstephalf}%
\newfont{\smallthickmath}{msbm7 scaled \magstephalf}%
\newfont{\footnotethickmath}{msbm8}%
\newfont{\footnotesmallthickmath}{msbm6}%

\newcommand{\s}{\sigma}        %

\theoremstyle{plain}             
\newtheorem{lemma}{\bf Lemma}[section]
\newtheorem{theorem}[lemma]{\bf Theorem}
\newtheorem{proposition}[lemma]{\bf Proposition}
\newtheorem{corollary}[lemma]{\bf Corollary}

\theoremstyle{definition}             
\newtheorem{definition}[lemma]{\bf Definition}

\theoremstyle{remark}             
\newtheorem{remark}[lemma]{\bf Remark}
\newtheorem*{remark*}{\bf Remark}

\newtheorem{fact}[lemma]{\bf Observation}

\newcommand{\Fo}{\,\,\,\text{for }\,\,}
\newcommand{\Foa}{\,\,\,\text{for all }\,\,}

\newcommand{\stint}{\mathop{\int\mskip-18.5mu-}}
\newcommand\Reals{{\mathbb R}}
\newcommand\R{{\mathbb R}}
\newcommand\Z{{\mathbb Z}}

\newcommand\N{{\mathbb N}}
\renewcommand\S{{\mathbb S}}

\newcommand\Sphere{{\mathbb S}}

\newcommand{\bbbr}{\Reals}

\newcommand\dist{\mathop{\rm dist}\nolimits}
\newcommand\diam{\mathop{\rm diam}\nolimits}

\newcommand\Span{\mathop{\rm span}\nolimits}

\newcommand\ang{\mathop{\mbox{$<\!\!\!)$}}\nolimits}

\newcommand\Id{{{\rm Id}}}

\newcommand{\F}{\mathscr{F}}
\newcommand{\E}{\mathscr{E}}

\newcommand{\V}{\mathscr{V}}

\newcommand{\DC}{K}
\newcommand{\GMC}{\mathcal{K}_G}
\newcommand{\GTP}{\mathcal{K}_\textnormal{tp}}
\newcommand{\GCi}{\mathcal{K}^{(i)}}
\newcommand{\HD}{d_{\mathscr{H}}}
\newcommand{\Ball}{\mathbb{B}}
\newcommand{\CBall}{\overline{\mathbb{B}}}
\newcommand{\hmin}{\mathfrak{h}_{\text{min}}}
\newcommand{\height}{\mathfrak{h}}
\newcommand{\face}{\mathfrak{fc}}
\newcommand{\ntheta}{\theta}
\newcommand{\nbeta}{\beta}

\DeclareMathOperator{\graph}{Graph}
\DeclareMathOperator{\conv}{conv}
\DeclareMathOperator{\aff}{aff}
\DeclareMathOperator{\lin}{span}
\DeclareMathOperator{\dgras}{\varangle}

\renewcommand{\R}{\bbbr}

\newcommand{\eps}{\varepsilon}
\renewcommand{\H}{\mathscr{H}}
\newcommand{\todo}[1]{}

\def\osc{\mathop{\rm osc\,}}         
\def\rtp{R_{\rm tp}}

\begin{document}

\renewcommand{\thefigure}{\arabic{figure}}


\title{\large\bf Characterizing  $W^{2,p}$~submanifolds by $p$-integrability of global curvatures}

\author{\normalsize S\l{}awomir Kolasi\'{n}ski\thanks{Partially supported by the
    MNiSzW Research project~N N201 611140.},\quad Pawe\l{}
  Strzelecki\thanks{Partially supported by the DFG--MNiSzW research project
    \emph{Geometric curvature energies (Mo966/4-1)}.},\quad Heiko von der
  Mosel\footnotemark[2]}

\date{\normalsize \today}

\maketitle


\frenchspacing

\vspace*{-.1cm}

\begin{abstract}
  We give sufficient and necessary geometric conditions, guaranteeing that an
  immersed compact closed manifold $\Sigma^m\subset \R^n$ of class $C^1$ and of
  arbitrary dimension and codimension (or, more generally, an Ahlfors-regular
  compact set $\Sigma$ satisfying a mild general condition relating the size of
  holes in $\Sigma$ to the flatness of $\Sigma$ measured in terms of beta
  numbers) is in fact an \emph{embedded\/} manifold of class $C^{1,\tau}\cap
  W^{2,p}$, where $p>m$ and $\tau=1-m/p$. The results are based on a careful
  analysis of Morrey estimates for integral curvature--like energies, with
  integrands expressed geometrically, in terms of functions that are designed to
  measure either (a) the shape of simplices with vertices on $\Sigma$ or (b) the
  size of spheres tangent to $\Sigma$ at one point and passing through another
  point of $\Sigma$.

  Appropriately defined \emph{maximal functions} of such integrands turn out to
  be of class $L^p(\Sigma)$ for $p>m$ if and only if the local graph
  representations of $\Sigma$ have second order derivatives in $L^p$ \emph{and}
  $\Sigma$ is embedded. There are two ingredients behind this result. One of
  them is an equivalent definition of Sobolev spaces, widely used nowadays in
  analysis on metric spaces. The second one is a careful analysis of local
  Reifenberg flatness (and of the decay of functions measuring that flatness)
  for sets with finite curvature energies. In addition, for the geometric
  curvature energy involving tangent spheres we provide a nontrivial lower bound
  that is attained if and only if the admissible set $\Sigma$ is a round sphere.
  \vspace{2mm}

\end{abstract}

\setcounter{tocdepth}{3}

\bigskip 


\renewcommand\theequation{{\thesection{}.\arabic{equation}}}
\def\setnumbers{\setcounter{equation}{0}}



\section{Introduction}\label{sec:1}

In this paper we address the following question: under what circumstances is a
compact, $m$-dimensional set $\Sigma$ in $\R^n$, satisfying some mild additional
assumptions, an $m$-dimensional embedded manifold of class $W^{2,p}$? For
$p>m=\dim\Sigma$ we formulate two necessary and sufficient criteria for a
positive answer.  Each of them says that $\Sigma$ is an embedded manifold of
class $W^{2,p}$ if and only if a certain geometrically defined integrand is of
class $L^p$ with respect to the $m$-dimensional Hausdorff measure on $\Sigma$.
One of these integrands measures the flatness of all $(m+1)$-dimensional
simplices with one vertex at a fixed point of $\Sigma$ and other vertices
elsewhere on $\Sigma$; see Definition \ref{def:1.2}.  The other one measures the
size of all spheres that touch an $m$-plane passing through a fixed point of
$\Sigma$ and contain another (arbitrary) point of $\Sigma$ (Definition
\ref{def:1.3}).

The extra assumptions we impose on the set $\Sigma$ are: (1) Ahlfors regularity
with respect to the $m$-dimensional Hausdorff measure $\H^m$, and (2) roughly
speaking, a certain relation between the flatness of $\Sigma$ and the size of
``holes'' it might have: the flatter $\Sigma$ is, the smaller these holes must
be.  To state the main result, Theorem \ref{mainthm}, formally, let us first
specify these two conditions precisely and then define the geometric integrands
mentioned above.  Throughout the paper we denote with $\Ball^n(a,s)$ an {\it
  open} $n$-dimensional ball of radius $s$ centered at the point $a\in\R^n$, and
we write $a\approx b$ if $a/C\le b\le Ca$ for some constant $C\ge 1$, and $
a\lesssim b$ (or $a\gtrsim b$), if only the left (or right) of these
inequalities holds.

\subsection{Statement of results}

\begin{definition}[\textbf{the class of $m$-fine sets}] \label{def:fine} Let
  $\Sigma \subset \R^n$ be compact.  We call $\Sigma$ an \emph{$m$-fine set} and
  write $\Sigma \in \F(m)$ if there exist constants $A_{\Sigma} > 0$ and
  $M_{\Sigma} \ge 2$ such that
  \begin{enumerate}
    \renewcommand{\labelenumi}{(\roman{enumi})}
  \item \textbf{(Ahlfors regularity)}
    \label{fine:reg}
    for all $x \in \Sigma$ and $r \le\diam\Sigma$ we have
    \begin{equation}
      \label{eq:fine:reg}
      \H^m(\Sigma \cap \Ball^n(x,r)) \ge A_{\Sigma} r^m\, ;
    \end{equation}
  \item \textbf{(control of ``holes'' in small scales)}
    \label{fine:gaps}
    for each $x \in \Sigma$ and $r \le \diam\Sigma$ we have
    \begin{displaymath}
      \ntheta_\Sigma(x,r) \le M_{\Sigma} \, \nbeta_\Sigma(x,r) \,.
    \end{displaymath}
  \end{enumerate}
\end{definition}

\reallyinvisible{\zz}{ {\tt \zz This Definition is a little bit problematic. If the
    second condition is to be satisfied for all $r \le \diam\Sigma$ then later
    we cannot make $M_{\Sigma}$ be an absolute constant as we claim in
    Remark~\ref{rem:Msigma} because we can make $\diam\Sigma$ as big as we want
    without increasing the curvature energy $\E_p$. To see this take a smooth
    cigar (cylinder with a spherical cap on each end) of length e.g. $9^{9^9}$
    and cross-sectional diameter $1$. For the point at the end of the cigar and
    for big $r$ the $\theta$ numbers can get arbitrarily close to $\frac 12$. I
    modified the arguments in Remark~\ref{rem:Msigma} and I believe now
    everything works.} }

\reallyinvisible{\yy}{ {\tt\xx do we need $\H^m(\Sigma)<\infty$?  cf. Hahlomaah's
    generalization of L\'eger's result...}  We can admit $\Sigma$ of infinite
  measure; a posteriori the measure must be finite.}

Here, $\beta_\Sigma$ and $\theta_\Sigma$ denote, respectively, the {\it beta
  numbers} and the {\it bilateral beta numbers} of $\Sigma$, defined by
\begin{align}
  \label{def:beta}
  \nbeta_\Sigma(x,r) &:= \frac 1r \inf \left\{
    \sup_{z \in \Sigma \cap \Ball(x,r)} \dist(z,x+H) : H \in G(n,m)
  \right\}, \\
  \label{def:theta}
  \ntheta_\Sigma(x,r) &:= \frac 1r \inf\Bigl\{
  \HD(\Sigma \cap \CBall(x,r), (x + H) \cap \CBall(x,r)) : H \in G(n,m)
  \Bigr\} \, ,
\end{align}
where $G(n,m)$ stands for the Grassmannian of all $m$-dimensional linear subspaces of $\R^n$, and where 
\begin{displaymath}
  \HD(E,F) := \sup\{ \dist(y,F) : y \in E \} + \sup\{ \dist(z,E) : z \in F \} 
\end{displaymath}
is the Hausdorff distance of sets in $\R^n$.  Intuitively, condition (ii) of
Definition~\ref{def:fine} ascertains that \emph{if} $\Sigma$ is flat at some
scale $r>0$, then the gaps and holes in $\Sigma$ cannot be large. Their sizes
are at most comparable to the degree of flatness of $\Sigma$.  If an $m$-fine
set $\Sigma$ satisfies $\beta_\Sigma(x,r)\to 0$ uniformly w.r.t $x\in\Sigma$ as
$r\to 0$, then $\Sigma$ is {\it Reifenberg flat with vanishing constant}, see
e.g. G. David, C. Kenig and T. Toro \cite[Definition 1.3]{davidkenigtoro} for a
definition.  However, note that neither the Reifenberg flatness of $\Sigma$, nor
\emph{rectifiability\/} of $\Sigma$ itself is required in
Definition~\ref{def:fine}. Both these properties \emph{follow} from the
finiteness of geometric curvature energies we consider here.

It is relatively easy to see that $\F(m)$ contains immersed $C^1$ submanifolds
of $\R^n$ (cf. \cite[Example 1.60]{slawek-phd} for a short proof), or embedded
Lipschitz submanifolds without boundary. It also contains other sets such as the
following stack of spheres $\Sigma=\bigcup_{i=0}^\infty\Sigma_i\cup\{0\}$, where
the $2$-spheres $\Sigma_i=\S^2(c_i,r_i)\subset\R^3$ with radii $r_i=2^{-i-2}>0$
are centered at the points $c_i=(p_i+p_{i+1})/2$ for $p_i=(2^{-i},0,0)\in \R^3$,
$i=0,1,2,\ldots.$ Note that the spheres $\Sigma_i$ and $\Sigma_{i+1}$ touch each
other at $p_{i+1}$, and the whole stack $\Sigma$ is an admissible set in the
class $ \F(2)$; see Figure \ref{Fig1}.

\begin{figure}[!h]  
  \begin{center}  
    \includegraphics*[totalheight=5cm]{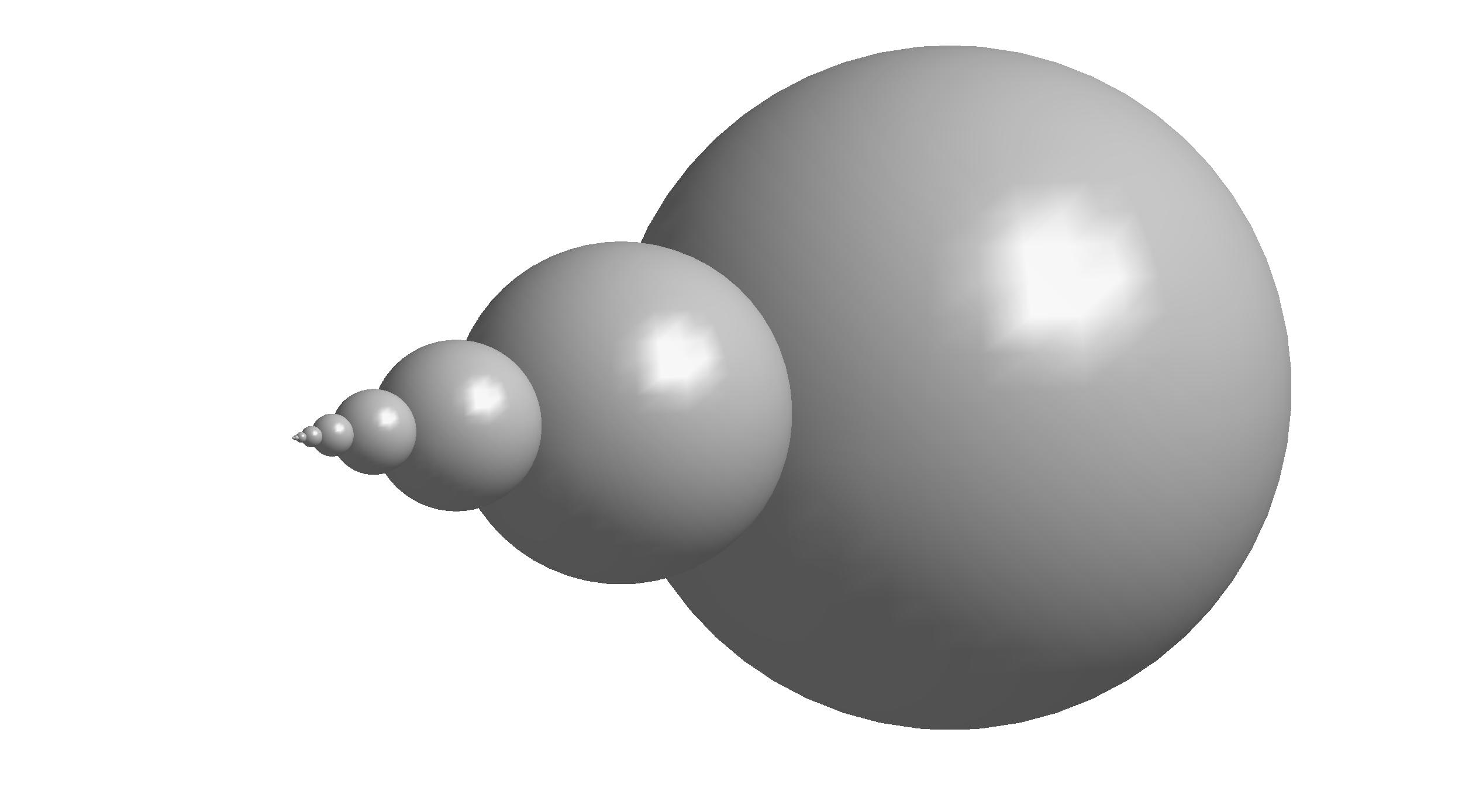} 
    \quad
    \includegraphics*[totalheight=5.8cm]{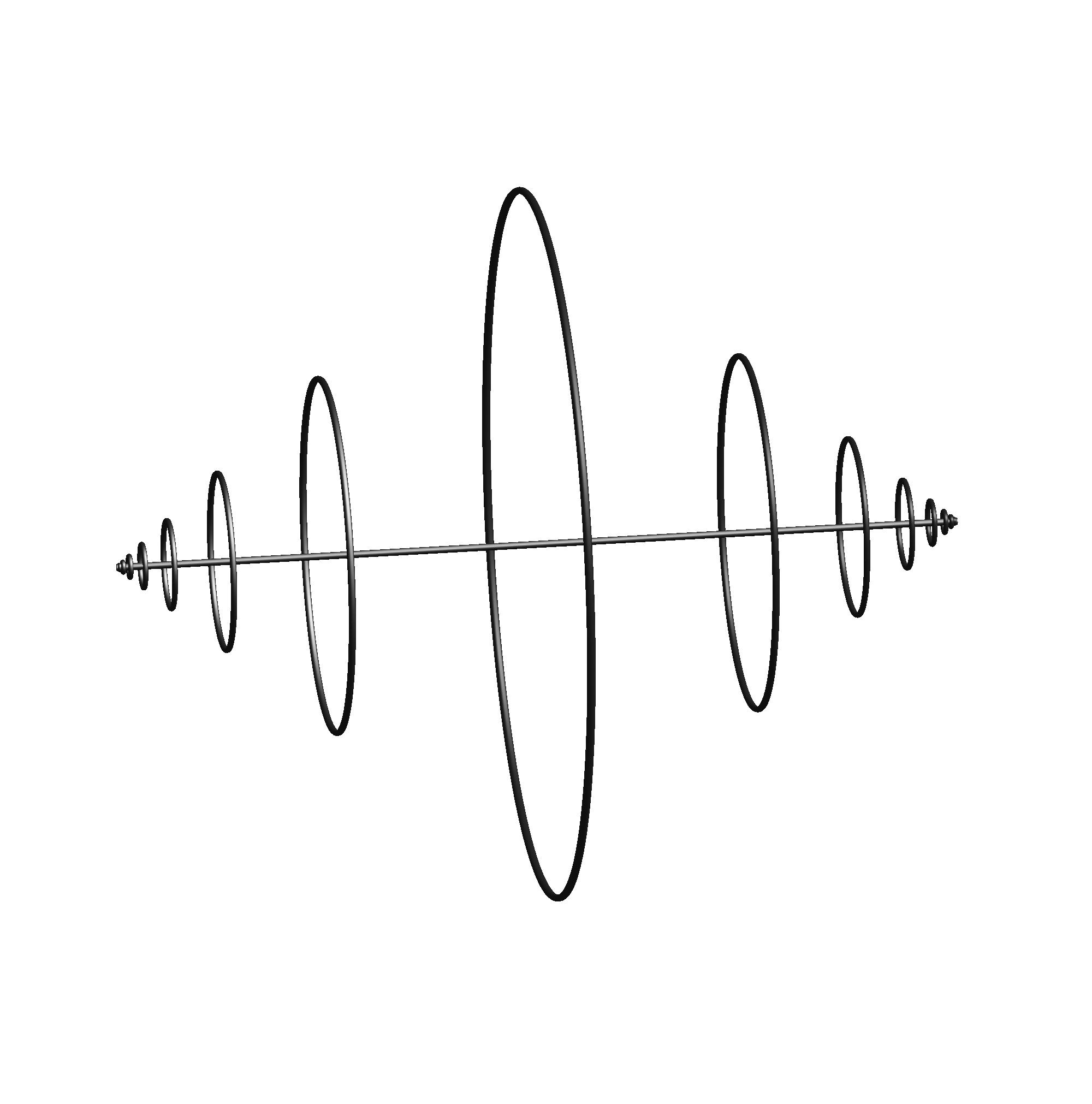}
  \end{center}
  
  \caption{\label{Fig1} Left: a union of countably many spheres is in $\F(2)\cap
    \mathscr{A}(\delta)$. Right: a set in $\F(1)\setminus \mathscr{A}(\delta)$.}
\end{figure}

A slightly different class $\mathscr{A}(\delta)$ of admissible sets was used by
the second and third author in \cite{svdm-tp1}. Roughly speaking, the elements
of $\mathscr{A}(\delta)$ are Ahlfors regular unions of countably many continuous
images of closed manifolds, and have to satisfy two more conditions: a certain
degree of flatness and a related linking condition; all this holds up to a set
of $\H^m$-measure zero. The class $\mathscr{A}(\delta)$ contains, for example,
finite unions of $C^1$ embedded manifolds that intersect each other along sets
of $\H^m$-measure zero (such as the stack of spheres in Figure \ref{Fig1}), and
bi-Lipschitz images of such unions, but also certain sets with cusp
singularities. For example, an arc with two tangent segments,
\[
A=\big\{x\in \R^2\colon x_1,x_2\ge0 \text{ and } \big(x_1^2+x_2^2=
1\text{ or }\max_{i=1,2} |x_i|=1\big)\big\} 
\]
is in $\mathscr{A}(\delta)$ for each $\delta>0$. However, $A$ is not in $\F(1)$
as the $\beta_A(\cdot,r)$ goes to zero as $r\to 0$ at the cusp points while
$\theta_A(x,r)$ remains constant there. On the other hand, the union of a
segment and countably many circles that are contained in planes perpendicular to
that segment,
\[
\big\{(t,0,0)\colon t\in [0,1]\big\}\cup \bigcup_{j=1}^\infty 
\gamma_j \cup \bigcup_{j=2}^\infty\tilde\gamma_j ,
\]
where
\[
\gamma_j=\big\{ 2^{-j}(1 , \cos \varphi,\sin\varphi)
\colon 
\varphi\in [0,2\pi]\big\}
\]
and $\tilde\gamma_j$ is the image of $\gamma_j$ under the reflection
$(x,y,z)\mapsto (1-x,y,z)$, is not in $\mathscr{A}(\delta)$ as the linking
condition is violated at all the points of the segment but it does belong to
$\F(1)$, as the circles prevent the $\beta(x,r)$ from going to zero at the
endpoints of the segment.

Both $\F(m)$ and $\mathscr{A}(\delta)$ contain sets of fractal dimension,
e.g. sufficiently flat von Koch snowflakes.  However, if one of our curvature
energies of $\Sigma$ is finite, it follows rather easily that the Hausdorff
dimension of $\Sigma$ must be $m$.


\begin{definition}[\textbf{Global Menger curvature at a point}]
\label{def:1.2}
  Let $\Sigma\in \mathcal{F}(m)$ and $x\in \Sigma$. Set
  \[
  \GMC[\Sigma](x)\equiv\GMC(x):=
 \sup_{x_1,\ldots, x_{m+1} \in \Sigma} \DC(x,x_1,\ldots, x_{m+1}) \, ,
  \]
  where
  \begin{equation}\label{Menger_simplex}
    \DC(x, x_1, \ldots, x_{m+1}):=
 \frac{\H^{m+1}(\conv(x ,x_1,\ldots, x_{m+1}))}{\diam \bigl(\{x ,x_1,\ldots, x_{m+1}\}\bigr)^{m+2}},
  \end{equation}
  and $\conv(E)$ and $\diam (E)$ denote the convex hull and the diameter of a
  set $E$, respectively\footnote{The function in \eqref{Menger_simplex}
    resembles the type of discrete curvatures considered by G. Lerman and
    J.T. Whitehouse \cite{LW08a}, \cite{LW08b} but scales differently, see
    Remark~5.2 in \cite{svdm-surfaces}. }.  We say that $\GMC(x)$ is the
  \emph{global Menger curvature of \,$\Sigma$\, at $x$.}
\end{definition}
When $m=1$ and $\Sigma$ is just a curve or a more general one-dimensional set\,
then $\DC(x_0,x_1,x_2)$ is the ratio of the area of the triangle
$T=\conv(x_0,x_1,x_2)$ to the third power of the maximal edge length of $T$.
Thus, $\DC$ is controlled by $R(T)^{-1}$, where $R(T)$ is the circumradius
of~$T$;
\[
\DC(x_0,x_1,x_2)\le \frac{1}{4R(T)}=\frac{\text{Area}\, (T)}{|x_0-x_1|\, |x_1-x_2|\, |x_2-x_0|}\, .
\]
For triangles with angles bounded away from $0$ and $\pi$, both quantities are
in fact comparable.  Therefore, in this case our global curvature function
$\GMC$ does not exceed a constant multiple of the \emph{global curvature} as
defined by O. Gonzalez and J.H. Maddocks \cite{GM}, and widely used afterwards;
see e.g. \cite{GMSvdM}, \cite{CKS02}, \cite{heiko1}, \cite{heiko2},
\cite{heiko3}, \cite{stvdm-MathZ}, \cite{gerlachvdm1}, \cite{gerlachvdm2}, and
for global curvature on surfaces \cite{StvdM1}, \cite{StvdM2}. Also for $m=2$,
integrated powers of a function quite similar to $K(x_0,x_1,x_2)$ in
\eqref{Menger_simplex} were used in \cite{svdm-surfaces} to prove geometric
variants of Morrey-Sobolev imbedding theorems for compact two-dimensional sets
in $\R^3$ in an admissibility class slightly more general than the class
$\mathscr{A}(\delta)$ defined in \cite{svdm-tp1}.

To define the second integrand, we first introduce the \emph{tangent-point
  radius}, which for the purposes of this paper is a function
\[
\rtp\colon \Sigma\times\Sigma \times G(n,m) \to [0,+\infty]
\]
given by
\begin{equation}\label{rtp}
  \rtp (x,y;H):=\frac{|y-x|^2}{2\dist(y,x+H)}\, .
\end{equation}
Geometrically, this is the radius of the smallest sphere tangent to the affine
$m$-plane $x+H$ and passing through~$x$ and~$y$. (If $y$ happens to be
contained in $x+H$, in particular if $y=x$, then we set $1/\rtp (x,y;H)=0$.)

\begin{definition}[\textbf{Global tangent-point curvature}]
\label{def:1.3}
  Assume that $H\colon \Sigma\to G(n,m)$ is an arbitrary map. Set
  \[
  \GTP[\Sigma](x)\equiv\GTP(x)\equiv \GTP(x,H(x)) :=
  \sup_{y\in \Sigma} \frac{1}{\rtp (x,y;H(x))}\, .
  \]
\end{definition}
Of course, the definition of $\GTP\colon \Sigma\to [0,+\infty]$ depends on the choice of $H$. However, we shall often omit the particular map $H$ from the notation, assuming tacitly that a choice of `tangent' planes $\Sigma\ni x\mapsto H(x)\in G(n,m)$ has been fixed. 

\begin{theorem}\label{mainthm}
  Let $0<m<n$ and $\Sigma\in \mathcal{F}(m)$. Assume $p>m$. The following conditions are equivalent:
  \begin{enumerate}
    \renewcommand{\labelenumi}{{\rm (\arabic{enumi})}}
  \item $\Sigma$ is an embedded $W^{2,p}$-submanifold of $\R^n$ without boundary;
  \item $\GMC[\Sigma] \in L^{p}(\Sigma,\H^m)$;
  \item There is a map $H\colon \Sigma\to G(n,m)$ such that for this map
    $$\GTP[\Sigma] \equiv \GTP(\cdot, H(\cdot)) \ \in  \   L^{p}(\Sigma,\H^m).$$
  \end{enumerate}
\end{theorem}
A quick comment on the equivalence of (1) and (3) should be made
right away: it is a relatively simple exercise to see that for a $C^1$ embedded manifold $\Sigma$ the $L^p$ norm of $\GTP(\cdot, H(\cdot))$ can be finite for \emph{at most} one continuous map
$H\colon\Sigma\to G(n,m)$ -- the one sending every $x\in \Sigma$ to $T_x\Sigma\in G(n,m)$.

Let us also mention a toy case of the equivalence of conditions (1) and (2) 
in the above theorem. For rectifiable curves $\gamma$ in $\R^n$ 
the equivalence of the arc-length parametrization $\Gamma$ of $\gamma$ being injective and in $W^{2,p}$,
and the global curvature of $\gamma$
being in $L^p$ has been proved by the second and third author in \cite{stvdm-MathZ}. To be more precise, let $S_L:=\R/L\Z,$ $L>0$, be the circle
with perimeter $L$, and denote by $\Gamma:S_L\to\R^n$ the
arclength parametrization of a closed rectifiable curve
$\gamma:\S^1\to\R^n$ of length $L$. Then the {\it global radius of curvature
  function} $\rho_G[\gamma]:S_L\to\R$, ; see, e.g., \cite{GMSvdM},   is defined as
\begin{equation}\label{global}
  \rho_G[\gamma](s):=\inf_{\s,\tau\in S_L\setminus\{s\}\atop
    \s\not=\tau}R(\Gamma(s),\Gamma(\s),\Gamma(\tau)),\qquad s\in S_L
\end{equation}
where, again, $R(\cdot,\cdot,\cdot)$ denotes the circumradius of a triangle, and the {\it global curvature}
$\kappa_G [\gamma](s)$ of $\gamma$ is given by
\begin{equation}\label{globalcurvature}
  \kappa_G[\gamma](s):=\frac{1}{\rho_G[\gamma](s)}\, .
\end{equation}
In \cite{stvdm-MathZ} we prove for $p>1$
that $\Gamma\in W^{2,p}(S_L,\R^n)$ and $\Gamma$ is injective (so that $\gamma$ is simple) if and only if $\kappa_G[\gamma]\in L^p$. Examples show that this 
fails for $p=1=\dim\gamma$: There are embedded curves of class
$W^{2,1}$ whose global curvature $\kappa_G$ is {\it not}
in $L^1.$   The first part of the proof (3) $\Rightarrow $ (1) for $m=1$, namely the optimal $C^{1,\tau}$-regularity of curves with finite energy, is modelled on the argument that was used in  \cite{svdm-tpcurves} for a different 
geometric 
curvature energy, namely for $\iint_{\gamma\times\gamma} 1/\rtp^q$.

We conjecture that the implications (1) $\Rightarrow$ (2), (3) of Theorem~\ref{mainthm} fail for $p=m>1$. 

\reallyinvisible{}{\tt\xx SHOULD we mention \cite{svdm-tpcurves} where at least
  part of the implication (3) $\Rightarrow $ (1), namely the optimal
  $C^{1,\tau}$-regularity was shown for $m=1$??\xx}

\begin{remark*}If (2) or (3) holds, then according to Theorem
\ref{mainthm} $\Sigma$ is embedded and locally, for some $R>0$, $\Sigma\cap \Ball^n(x,R)$ is congruent to a graph of a $W^{2,p}$ function $f\colon \R^m\to\R^{n-m}$. Since $p>m$, we also know from a result of A. Calder\'{o}n and A. 
Zygmund (see e.g. \cite[Theorem 1, p. 235]{EG}) that $Df\colon \R^m\to L(\R^m,\R^{n-m})$ is differentiable a.e. in the classic
sense.
\end{remark*}

\begin{remark*}
One can complement Theorem 1.4 by the contribution of S. Blatt and the 
first author \cite{blatt-k} in
the following way.  Suppose that $2\le k\le m+2$ and in Definition~1.2 one takes the supremum only with respect to $(m+2)-k$ points of $\Sigma$, defining the respective curvature $\mathcal{K}_{G,k}$ as a function of $k$-tuples $(x_0,x_1,\ldots,x_{k-1})\in \Sigma^k$.  Suppose that $p>m(k-1)$ and $\Sigma$ is a $C^1$ embedded manifold.
Then, $\mathcal{K}_{G,k}$ is of class $L^p(\Sigma^k,\H^{mk})$ if and only if $\Sigma$ is locally a graph of class $W^{1+s,p}(\R^m,\R^{n-m})$, where $s =
  1-m(k-1)p^{-1}\in (0,1)$.  If $k=m+2$ and $p>m(m+2)$, then the assumption that $\Sigma$ be a
  $C^1$ manifold is not necessary; one can just assume $\Sigma\in \F(m)$.  See
  \cite{blatt-k} for details.  We believe that the characterization of
  \cite{blatt-k} does hold for all $2\le k\le m+2$ without the assumption that
  $\Sigma$ is of class $C^1$. (To prove this, one would have to generalize the
  regularity theory presented in \cite{slawek-phd}  to all curvatures
  $\mathcal{K}_{G,k}$). 

Blatt's preprint \cite{blatt-tp} contains a similar characterization in
terms of fractional Sobolev spaces of those $C^1$ manifolds $\Sigma$ for which the tangent--point energy $\iint_{\Sigma\times\Sigma} 1 / (\rtp)^q $ is finite.
\end{remark*}

\begin{remark*}
  W. Allard, in his classic paper \cite{allard}, develops a regularity theory
  for $m $-dimensional varifolds whose first variation (i.e., the distributional
  counterpart of mean curvature) is in $L^p$ for some $p>m $. His Theorem 8.1
  ascertains that, under mild extra assumptions on the density function of such
  a varifold $V$, an open and dense subset of the support of $\|V\|$ is locally
  a graph of class $C^{1,1-m/p}$.  For $p>m$ Sobolev--Morrey imbedding yields
  $W^{2,p}\subset C^{1,1-m/p}$ and one might na\"{\i}vely wonder if a stronger
  theorem does hold, implying Allard's (qualitative) conclusion just by
  Sobolev--Morrey. Indeed, J.P. Duggan \cite{duggan} proved later an optimal result
  in this direction.  For integral varifolds, $W^{2,p}$-regularity can be
  obtained directly via elliptic regularity theory, see U. 
  Menne \cite[Lemmata
  3.6 and~3.21]{menne}.

  In Allard's case the `lack of holes' is built into his assumption on the first
  variation $\delta V$ of $V$.  Our setting is not so close to 
  PDE theory:
  both `curvatures' are defined in purely geometric terms and in a nonlocal way.
  Here, the `lack of holes' follows, roughly speaking, from a delicate interplay
  between the inequality $\theta(x,r)\lesssim \beta(x,r)$ built into the
  definition of $\F(m)$ and the decay of $\beta(x,r)$ which follows from the
  finiteness of energy. A more detailed account on our strategy of proof here,
  is presented in the next subsection.

  At this stage we do not know for our curvature energies what the situation is
  like in the scale invariant case $p=m$. For two-dimensional integer
  multiplicity varifolds, however (or in the simpler situation of
  $W^{2,2}$-graphs over planar domains)  
  Toro \cite{toro} was able to prove
  the existence of bi-Lipschitz parametrizations. For $m$-dimensional sets Toro
  \cite[eq. (1)]{toro2} established a sufficient condition for the existence
  of bi-Lipschitz parametrizations in terms of $\theta$.  Her condition is
  satisfied, e.g., by S.~Semmes' chord-arc surfaces with small constant, and by
  graphs of functions that are sufficiently well approximated by affine
  functions; see \cite[Section 5]{toro2} for the details.
\end{remark*}

\begin{remark*}
Following the reasoning in \cite[Lemma 7]{stvdm-MathZ} one can easily provide
nontrivial lower bounds for the global tangent-point curvature for hypersurfaces 
($n=m+1$), and also for curves $m=1<n$; see Theorem \ref{thm:1.5}
below. Indeed, setting $E:=\|\GTP[\Sigma]\|_{L^p(\Sigma)}$, where $\Sigma\subset\R^n$ is a 
compact  connected $m$-dimensional $C^1$-submanifold without
boundary, we can
find at least one point $x\in\Sigma$ such that 
$\GTP[\Sigma](x)\le E/(\mathscr{H}^m(\Sigma)^{1/p}),$
since otherwise we had a contradiction via
$$
E=\left(\int_\Sigma\Big(\GTP[\Sigma](x)\Big)^p\,d\mathscr{H}^m(x)\right)^{1/p}>\frac{E}{\mathscr{H^m}(\Sigma)^{1/p}}
\mathscr{H^m}(\Sigma)^{1/p}=E.
$$
Therefore $R:=\inf_{y\in\Sigma} R_{\textnormal{tp}}(x,y,T_x\Sigma)\ge\mathscr{H}^m(\Sigma)^{1/p}/E.$ If there existed
an open ball $\Ball^n(a,R)$ with 
$$
(x+T_x\Sigma)\cap\partial\Ball^n(a,R)=\{x\}
$$
such that $\Sigma\cap\Ball^n(a,R)\not=\emptyset$, then we could find a strictly
smaller sphere tangent to $\Sigma $ in $x$ and containing yet another point 
$y\in\Sigma$ contradicting the definition of $R$.
Hence we have shown that the union of such open balls
\begin{equation}\label{MDef}
M:=\bigcup\{\Ball^n(a,R):\partial\Ball^n(a,R)\cap (x+T_x\Sigma)=\{x\}\}
\end{equation}
contains no point of $\Sigma.$
In other words, $\Sigma$ is a compact  embedded submanifold  without boundary, contained in
$\R^n\setminus M$, and one can ask for the area minimizing submanifold in $\R^n\setminus M$.
In codimension one, i.e., for $m=n-1$,  $\Sigma=\partial \Omega$ for a bounded open set $\Omega\subset \R^n$, and
the union of balls defining $M$ just consists of two such balls,
one in $\Omega$ and one in the unbounded
exterior of $\Sigma$. So, due to the classic
isoperimetric inequality (see, e.g. \cite[Theorem 3.2.43]{federer}) one finds
\begin{eqnarray*}
\mathscr{H}^{n-1}(\Sigma) & \ge &
n\omega_n^{1/n}\mathscr{H}^{n}((\Omega))^{\frac{n-1}{n}}\\
& \ge &
n\omega_n^{1/n}\mathscr{H}^{n}(B(a,R))^{\frac{n-1}{n}}=
\mathscr{H}^{n-1}(\partial\Ball^n(a,R))=n\omega_nR^{n-1}.
\end{eqnarray*}
which by definition of $R$ can be rewritten as
\begin{equation}\label{lower_energy_bound}
\|\GTP[\Sigma]\|_{L^p(\Sigma)}=E\ge (\mathscr{H}^{n-1}(\Sigma))^{\frac 1p -\frac{1}{n-1}}(n\omega_n)^{\frac{1}{n-1}}
\end{equation}
with equality if and only if $\Sigma$ equals a round
sphere.  Hence, we obtain the following simple result. 

\begin{theorem}\label{thm:1.5}
Let $p>0$. Among all compact embedded
$C^1$-hyper\-sur\-faces  with given surface area, the 
round sphere uniquely (up to isometries)
minimizes the
energy $\|\GTP[\Sigma]\|_{L^p(\Sigma,\mathscr{H}^{n-1})}.$ If $p>n-1$, the same holds true for all $(n-1)$-fine sets
$\Sigma\in\F(n-1)$.
\end{theorem}
Similarly, for $m=1$ one concludes {\it that any of those great circles on any of the balls
$\Ball^n(a,R)$ generating $M$ in \eqref{MDef} that are also geodesics on $M$ 
uniquely minimize
$E$ among all closed simple $C^1$-curves $\Sigma\equiv\gamma\subset\R^n\setminus M$,} which
provides the lower bound
\begin{equation}\label{lower_energy_curves}
\|\GTP[\gamma]\|_{L^p(\gamma)}=E\ge 2\pi\mathscr{H}^1(\gamma)^{\frac 1p -1}.
\end{equation}
This is exactly what we found for curves in \cite[Lemma 7 (3.1)]{stvdm-MathZ}, and is also consistent
with \eqref{lower_energy_bound} if $n=2=m+1$. 
\end{remark*}

\subsection{Essential ideas and an outline of the proof.}

This paper grew out of our interest in geometric curvature energies and earlier
related research, cf. \cite{stvdm-MathZ}, \cite{ssvdm-triple},
\cite{svdm-surfaces}, \cite{svdm-tp1} and \cite{slawek-phd}.  While working on
the integral Menger curvature energy of rectifiable curves $\gamma\subset\R^n$
\[
\mathscr{M}_p(\gamma)=\iiint_{\gamma\times\gamma\times\gamma} \frac{1}{R^p(x,y,z)}\, d\H^1(x)\, d\H^1(y)\, d\H^1(z)\, , \qquad p>3,
\]
we realized how slicing can be used to obtain optimal H\"older continuity of
arc-length parametrizations.\footnote{The second and the third author of this
  paper acknowledge with gratitude the stimulating conversations that they had
  in the spring of 2008 with Joan Verdera at CRM in Pisa.  His insight that most
  of the work in \cite{ssvdm-triple} should and could be phrased in the language
  of beta numbers has helped us a lot in our subsequent research.} (The scale
invariant exponent $p=3$ is critical here: polygons have infinite $M_p$-energy
precisely for $p\ge 3$; see S. Scholtes \cite{scholtes} for a proof).

One crucial difference between curves $\gamma$ and  $m$-dimensional sets
$\Sigma$ in $\R^n$ for $m\ge 2$ lies in the distribution of mass in balls on
various scales: If $\gamma$ is a~rectifiable curve and $r< \frac
12\diam\gamma$, then obviously $\H^1 (\gamma\cap \Ball^n(x,r))\ge r$  for
each $x\in\gamma$. For $m>1$ the measure $\H^m (\Sigma \cap \Ball^n(x,r))$ might
be much smaller than $r^m$ due to complicated geometry of $\Sigma$ at
intermediate length scales. In \cite{svdm-surfaces} we have devised a method,
allowing us to obtain estimates of $\H^m (\Sigma \cap \Ball^n(x,r))$ for $m=2$,
$n=3$ and all radii $r<R_0$, with $R_0$ depending only on the energy level of
$\Sigma$ in terms of its integral Menger curvature.  This method has been later
reworked and extended in the subsequent papers \cite{svdm-tp1},
\cite{slawek-phd}, to yield the so-called {\it uniform Ahlfors regularity},
i.e., estimates of the form
\[
\H^m (\Sigma \cap \Ball^n(x,r))\ge \frac 12 \omega_mr^m, \qquad \Foa r<R_0=R_0(\text{energy})\, ,
\]
for other curvature energies and arbitrary $0<m<n$ (to cope with the case of
higher codimension, we used a linking invariant to guarantee that $\Sigma$ has
large projections onto some $m$-dimensional planes). Combining such estimates
for $\H^m (\Sigma \cap \Ball^n(x,r))$ with an extension of ideas from
\cite{ssvdm-triple} we obtained in \cite{svdm-surfaces}, \cite{svdm-tp1} and
\cite{slawek-phd} a series of results, establishing $C^{1,\alpha}$ regularity
for surfaces, or more generally, for a priori non-smooth $m$-dimensional sets
for which certain geometric curvature energies are finite.  Finally, we also
realized that the well-known pointwise characterization of $W^{1,p}$-spaces of
P. Haj\l{}asz \cite{Haj96} is the missing link, allowing us to combine the ideas
from \cite{slawek-phd} and \cite{svdm-tp1} in the present paper in order to
provide with Theorem~\ref{mainthm} a far-reaching, general extension of
\cite[Theorems 1 \& 2]{stvdm-MathZ} from curves to $m$-dimensional manifolds in
$\R^n$.

\medskip

Let us now discuss the plan of proof of Theorem~\ref{mainthm} and outline the
structure of the whole paper.

The easier part is to check that if $\Sigma$ is an embedded compact $W^{2,p}$
manifold without boundary, then conditions (2) and (3) hold.  We work in small
balls $\Ball(x,R)$ centered on $\Sigma$, with $R>0$ chosen so that $\Sigma\cap
\Ball(x,R)$ is a (very flat) graph of a $W^{2,p}$ function $f\colon \Ball^m(x,2R)\to
\R^{n-m}$. Using Morrey's inequality  twice, we first show that
\[
\beta_\Sigma(a,r)\lesssim g(a)r, \qquad a\in \Ball(x,R)\cap\Sigma,\quad 0< r<R\, ,
\]
for a function $g\in L^p$ that is comparable to some maximal function of $|D^2
f|$. Next, working with this estimate of beta numbers on all scales
$r=R/2^k$, $k=0,1,2,\ldots$, we show that in each coordinate patch each of the
global curvatures $\GMC$ and $\GTP$ can be controlled by two terms,
\[
\GMC(a), \text{resp. }\GTP(a)\lesssim  g(a) + C(R)\, 
\]
where $C(R)$ is a harmless term depending only on the size of the patches.  (It
is clear from the definitions that for embedded manifolds one can estimate both
$\GMC$ and $\GTP$ taking into account only the local bending of $\Sigma$ and
working in coordinate patches of fixed size; the effects of self-intersections
are not an issue).  This yields $L^p$-integrability of $\GMC$ and $\GTP$.  We
refer to Section 4 for the details.

The reverse implications require more work. The proofs that (3) or (2) implies
(1) have, roughly speaking, four separate stages.  First, we use energy
estimates to show that if $\|\GMC\|_{L^p}$ or $\|\GTP\|_{L^p}$ are less than
$E^{1/p}$ for some finite constant $E$, then
\[
\beta_\Sigma(x,r) \lesssim \left(\frac{E}{A_\Sigma}\right)^{\kappa/(p-m)} r^{\kappa}\, .
\]
Here $\kappa$ denotes a number in $(0,1-m/p)$, depending only on $m,p$ with
different explicit values for $\GMC$ or $\GTP$, and $A_\Sigma$ is the 
constant from Definition \ref{def:fine}
measuring Ahlfors regularity of $\Sigma$.  By the very definition of $m$-fine
sets, such an estimate implies that the bilateral beta numbers of $\Sigma$ tend
to zero with a speed controlled by $r^{\kappa}$. In particular, $\Sigma$ is
Reifenberg flat with vanishing constant, 
and an application of \cite[Proposition 9.1]{davidkenigtoro}
shows that $\Sigma$ is an embedded manifold of class $C^{1,\kappa}$.  See
Section 3.1 for more details.

Next, we prove the uniform Ahlfors regularity of $\Sigma$, i.e. we show that
\[
\H^m(\Sigma\cap \Ball(x,r)) \ge \frac{1}{2}\H^m(\Ball^m(x,r))
\]
for all radii $r\in (0,R_0)$, where $R_0$ depends \emph{only\/} on the energy
bound $E$ and the parameters $n,m,p$, but not at all on $\Sigma$ itself.  Here,
we rely on methods from our previous papers \cite{slawek-phd} and
\cite{svdm-surfaces,svdm-tp1}.  Roughly speaking, we combine topological
arguments based on the linking invariant with energy estimates to show that for
each $r<R_0=R_0(E,n,m,p)$ the portion of $\Sigma$ in $\Ball^n(x,r)$ has large
projection onto some plane $H=H(r)\in G(n,m)$. See Section 3.2.

(There is a certain freedom in this phase of the proof; it would be possible to
prove uniform Ahlfors regularity first, and estimate the decay of
$\beta_\Sigma(x,r)$ afterwards. This approach has been used in
\cite{svdm-surfaces,svdm-tp1}.)

After the second step we know that in coordinate patches of diameter comparable
to $R_0$ the manifold $\Sigma$ coincides with a graph of a function $f\in
C^{1,\kappa}(\Ball^m,\R^{n-m})$.  The third stage is to bootstrap the H\"{o}lder
exponent $\kappa$ to the optimal $\tau=1-m/p>\kappa$ for both global curvatures
$\GMC$ and $\GTP$.  This is achieved by an iterative argument which uses
slicing: If the integral of the global curvature to the power $p$ over a ball is
not too large, then this global curvature itself cannot be too large on a
substantial set of \emph{good points\/} in that ball. Geometric arguments based
on the definition of the global curvature functions $\GMC$ and $\GTP$ show that
$|D f(x)-D f(y)|\lesssim |x-y|^\tau$ on the set of good points. It turns out
that there are plenty of good points at all scales, and in the limit we obtain a
similar H\"{o}lder estimate on the whole domain of $f$. See Section~3.3.

The fourth and last step is to combine the $C^{1,\tau}$-estimates with a
pointwise characterization of first order Sobolev spaces obtained by
Haj\l{}asz \cite{Haj96}. The idea is very simple. Namely, the bootstrap reasoning in the
third stage of the proof (Section 3.3) yields the following, e.g. for the global
Menger curvature $\GMC$: On a scale $R_1\approx R_0$, the intersection $\Sigma
\cap \Ball^n(a,R_1)$ coincides with a flat graph of a function $f\colon
P\simeq\R^m\to\R^{n-m}\simeq P^\perp$, with
\[
|Df(x)-D f(y)|\lesssim \biggl(\int_{\Ball^m(\frac{x+y}2, 5|x-y|)} \GMC\bigl((\xi,f(\xi))\bigr)^p\, d\xi
\biggr)^{1/p} |x-y|^\tau\, 
\]
for $\tau =1-m/p$. 
Such an inequality is true for \emph{every} $p>m$ so we can easily fix
a number $p'\in (m,p)$ and show that
\begin{equation}
  \label{prehajlasz}
  |D f(x)-D f(y)|\lesssim \bigl(M(x)+M(y)\bigr) |x-y|\, ,
\end{equation}
where $M(\cdot)^{p'}$ is the Hardy--Littlewood maximal function of the global
curvature. 
Since $p/p'>1$, an application of the Hardy--Littlewood
maximal theorem yields $M^{p'}\in L^{p/p'}$, or, equivalently, $M\in L^p$. Thus,
by the well known result of Haj\l{}asz (see Section 2.3), \eqref{prehajlasz}
implies that $D f\in W^{1,p}$. In fact, the $L^p$ norm of $D^2f$ is controlled
by a constant times the $L^p$-norm of the global Menger curvature $\GMC$.  An
analogous argument works for the global tangent-point curvature function $\GTP$.
This concludes the whole proof; see Section 3.4.

For each of the global curvatures $\mathcal{K}^{(i)}_G$, there are some
technical variations in that scheme; here and there we need to adjust an
argument to one of them. However, the overall plan is the same in both cases.

The  paper is organized as follows. In Section~2, we gather some
preliminaries  from linear algebra and some elementary facts about
simplices, introduce some specific notation, and list some
auxiliary results 
with references to existing literature. Section~3 forms the
bulk of the paper. Here, following the sketch given above, we prove that $L^p$
bounds for (either of) the global curvatures imply that $\Sigma$ is an embedded
manifold with local graph representations of class $W^{2,p}$. Finally, in
Section~4 we prove the reverse implications, concluding the whole proof of
Theorem~\ref{mainthm}.

\medskip\noindent\textbf{Acknowledgement.} The authors are grateful to the anonymous referee for her/his careful reading of this paper and the suggestions which have helped to improve the presentation of our work.

\section{Preliminaries}
\label{sec:2}

\subsection{The Grassmannian}

\def\red{(\rho,\eps,\delta)}

In this paragraph we gather a few elementary facts about the angular metric
$\dgras(\cdot,\cdot)$ on the Grassmannian $G(n,m)$ of $m$-dimensional linear
subspaces\footnote{Formally, $G(n,m)$ is defined as the homogeneous space
  \begin{displaymath}
    G(n,m) := O(n) / (O(m) \times O(n-m)) \,,
  \end{displaymath}
  where $O(n)$ is the orthogonal group; see e.g. A. 
  Hatcher's book~\cite[Section
  4.2, Examples 4.53, 4.54 and 4.55]{hatcher} for the reference. Thus $G(n,m)$
  could be treated as a topological space with the standard quotient
  topology.  Instead, we work with the angular metric $\dgras(\cdot,\cdot)$,
  see Definition \ref{def:grasmetric}.} of $\R^n$.

Here is a summary: for two $m$-dimensional linear subspaces
\[
U = \mathrm{span}\, \{ u_1, \ldots, u_m \}
\qquad\text{and}\qquad
V = \mathrm{span}\, \{ v_1, \ldots, v_m \}
\]
in $\R^n$ such that the bases $(u_1,\ldots,u_m)$, $(v_1,\ldots,v_m)$ are roughly
orthonormal and such that $|u_i - v_i| \le \varepsilon$, we have the estimate
$\dgras(U,V) \lesssim \varepsilon$. This will become especially useful in
Section~\ref{bootstrap}.


For $U\in G(n,m)$ we write $\pi_U$ to denote the orthogonal projection of $\R^n$
onto $U$ and we set $Q_U=\Id_{\R^n}-\pi_U=\pi_{U^\perp}$, where
$\Id_{\R^n}:\R^n\to\R^n$ denotes the identity mapping.
\begin{definition}\label{def:grasmetric}
  Let $U,V \in G(n,m)$. We set
  \begin{displaymath}
    \dgras(U,V) := \| \pi_U - \pi_V \| = \sup_{w \in \S^{n-1}} | \pi_U(w) - \pi_V(w) | \,.
  \end{displaymath}
\end{definition}
The function $\dgras(\cdot,\cdot)$ defines a metric on the Grassmannian
$G(n,m)$. The topology induced by this metric agrees with the standard quotient
topology of $G(n,m)$.  We list several properties of $\dgras$ below. They
will become useful for H\"{o}lder estimates of the graph parameterizations of
$\Sigma$ in Section~\ref{bootstrap}. The proofs are elementary and we omit them
here.

\begin{remark*}
   Notice that
  \begin{displaymath}
    \dgras(U,V) = \| \pi_U - \pi_V \| = \| \Id_{\R^n} - Q_U - (\Id_{\R^n} - Q_V) \| = \| Q_V - Q_U \| \,.
  \end{displaymath}
\end{remark*}

\begin{proposition}[Lemma 2.2 in \cite{svdm-tp1}]
  \label{prop:close-bases}
  If the spaces $U,V \in G(n,m)$ have orthonormal bases $(e_1,\ldots,e_m)$ and
  $(f_1,\ldots,f_m)$, respectively, and if $|e_i - f_i| \le \vartheta$ for $i
  = 1,\ldots,m$, then $\dgras(U,V) \le 2m\vartheta$.
\end{proposition}

\begin{definition}
  \label{ortho_basis}
  Let $V \in G(n,m)$ and let $(v_1,\ldots,v_m)$ be a basis of $V$. Fix some
  radius $\rho > 0$ and two constants $\varepsilon \in (0,1)$ and $\delta \in
  (0,1)$. We say that $(v_1,\ldots,v_m)$ is a \emph{$\red$-basis} if
  \begin{align*}
    (1-\varepsilon) \rho \le |v_i| &\le (1+\varepsilon) \rho \quad \text{for } i = 1, \ldots, m\\ 
    \text{and} \quad
    |\langle v_i, v_j \rangle| &\le \delta \rho^2 \quad \text{for } i \ne j \,.
  \end{align*}
  Specifically, a $(\rho,0,0)$-basis will be called {\it ortho-$\rho$-normal.}
\end{definition}

\begin{proposition}
  \label{prop:gs-red}
  Let $\rho > 0$, $\varepsilon \in (0,1/2)$ and $\delta \in (0,1)$ be some
  constants. Let $(v_1,\ldots,v_m)$ be a $\red$-basis of $V \in G(n,m)$. Then
  there exist an ortho-$\rho$-normal-basis $(\hat{v}_1,\ldots,\hat{v}_m)$ of $V$
  and a constant $C_2 = C_2(m)$ such that
  \begin{align*}
    |v_i - \hat{v}_i| \le (\varepsilon + C_2 \delta) \rho
    \quad \text{for } i = 1,\ldots,m \,.
  \end{align*}
\end{proposition}

\begin{proof}
  By scaling we may assume that $\rho=1$.
  \heikodetail{
    \bigskip
    
    If $\rho\not=1$ then scale $v_i/\rho$, find from the following proof an ONB
    $\{e_i\}$ of $V$ such that $|v_i/\rho-e_i|\le\epsilon+C_2(m)$ and then use
    $\rho\cdot e_i$ as desired $\hat{v}_i$.
    
    \bigskip
  }
  Define $w_i:=v_i/|v_i|$ for $i=1,\ldots,m$, $f_1:=w_1$, $\hat{v}_1:=w_1,$ and
  then recursively
  $$
  f_k:=w_k-\sum_{i=1}^{k-1}\langle w_k,\hat{v}_i\rangle \hat{v}_i,
  \qquad\textnormal{and}\quad \hat{v}_k:=f_k/|f_k|\quad\textnormal{
    for $k=1,\ldots,m$},
  $$
  and observe that $|w_i-v_i| = |1-|v_i|| \le \epsilon$  
  and $|\langle w_i,w_j\rangle| \le \delta/(1-\epsilon)^2 < 4\delta$ 
  for all $i,j=1,\ldots,m$, and in addition, 
  $V=\Span\{w_1,\ldots,w_m\}
  =\Span\{\hat{v}_1,\ldots,\hat{v}_m\}$ by construction.
  Notice that
  $||f_k|-1|=||f_k|-|w_k||\le|f_k-w_k|$, and therefore
  $$
  |f_k-\hat{v}_k|=||f_k|-1|\le|f_k-w_k|
  $$
  so that by 
  $$
  |v_k-\hat{v}_k|\le |v_k-w_k|+|w_k-f_k|+|f_k-\hat{v}_k|\le
  \epsilon+2|f_k-w_k|
  $$
  the main task turns out to be to estimate $a_k:=|f_k-w_k|$ for $k=1,\ldots,m$,
  where we get immediately $a_1=0$ by definition. If one estimates
  \begin{align*}
    a_k &\le \sum_{i=1}^{k-1}| \langle w_k,w_i \rangle| + \sum_{i=1}^{k-1}|w_i-\hat{v}_i| \\
    &\le 4\delta(k-1)+\sum_{i=1}^{k-1}(a_i+|f_i-\hat{v}_i|)\\
    &\le 4\delta(k-1)+2\sum_{i=1}^{k-1}a_i,
  \end{align*}
  one can prove by induction that
  \begin{displaymath}
    a_k \le 4\delta \Big[ (k-1)+2\sum_{i=0}^l3^i(k-i-2) \Big]
    + 2 \cdot 3^{l+1} \sum_{i=1}^{k-l-2} a_i
    \quad\text{for all } l=0,\ldots,k-3.
  \end{displaymath}
  Specifically for $l=k-3$ we obtain
  $$
  a_k \le 4\delta \Big[(k-1)+2\sum_{i=0}^{k-3}3^i(k-i-2)\Big],
  $$
  and therefore, for all $k=1,\ldots,m,$
  \begin{displaymath}
    |v_k-\hat{v}_k| \le \varepsilon + 8\delta\Big[(m-1) +
    2\sum_{i=0}^{m-3}3^i(m-i-2)\Big] =: \varepsilon + C_2(m)\delta \,.
    \qedhere
  \end{displaymath}
\end{proof}

\begin{proposition}
  \label{prop:dist-ang}
  Let $U,V \in G(n,m)$ and let $(e_1,\ldots,e_m)$ be some orthonormal basis of
  $V$. Assume that for each $i = 1,\ldots,m$ we have the estimate $\dist(e_i,U)
  = |Q_U(e_i)| \le \vartheta$ for some $\vartheta \in (0,1/\sqrt{2})$. Then
  there exists a constant $C_3 = C_3(m)$ such that
  \begin{displaymath}
    \dgras(U,V) \le C_3 \vartheta \,.
  \end{displaymath}
\end{proposition}

\begin{proof}
  Set $u_i := \pi_U(e_i)$. For each $i = 1, \ldots, m$ we have $|Q_U(e_i)| \le
  \vartheta$, so
  \begin{align}
    |u_i - e_i| &= |Q_U(e_i)| \le \vartheta
    \quad \text{hence} \notag \\
    \label{est:ui-len}
    1 - \vartheta^2 < \sqrt{1 - \vartheta^2} &\le |u_i| \le 1 < 1 + \vartheta^2
    \quad \text{for } i = 1, \ldots, m \,.
  \end{align}
  For any $i \ne j$ the vectors $e_i$ and $e_j$ are orthogonal, hence
  \begin{align*}
    0 = \langle e_i, e_j \rangle 
    &= \langle \pi_U(e_i) + Q_U(e_i), \pi_U(e_j) + Q_U(e_j) \rangle \\
    &= \langle \pi_U(e_i), \pi_U(e_j) \rangle + \langle Q_U(e_i), Q_U(e_j) \rangle \,.
  \end{align*}
  Therefore
  \begin{equation}
    \label{est:uiuj-ang}
    |\langle u_i, u_j \rangle| 
    = | \langle Q_U(e_i), Q_U(e_j) \rangle | 
    \le |Q_U(e_i)| |Q_U(e_j)| \le \vartheta^2 \,.
  \end{equation}
  Estimates \eqref{est:ui-len} and \eqref{est:uiuj-ang} show that
  $(u_1,\ldots,u_m)$ is a $\red$-basis of $U$ with constants $\rho = 1$,
  $\varepsilon = \vartheta^2$ and $\delta = \vartheta^2$. Let $(f_1,\ldots,f_m)$
  be the orthonormal basis of $U$ arising from $(u_1,\ldots,u_m)$ by means of
  Proposition~\ref{prop:gs-red}, so that we obtain
  \begin{displaymath}
    |f_i - e_i| \le |f_i - u_i| + |u_i - e_i|
    \le (1+C_2) \vartheta^2 + \vartheta \,.
  \end{displaymath}
  Using Proposition~\ref{prop:close-bases} and the fact that $\vartheta^2 <
  \vartheta < 1$ we finally get
  \begin{displaymath}
    \dgras(U,V) \le 2 m ((1+C_2) \vartheta^2 + \vartheta) 
    \le 2 m (1+C_2 + 1) \vartheta \,.
  \end{displaymath}
  Now we can set $C_3 = C_3(m) := 2 m (1+C_2(m) + 1) = 2m(2+C_2(m))$.
\end{proof}

\begin{proposition}
  \label{prop:red-ang}
  Let $(v_1,\ldots,v_m)$ be a $\red$-basis of $V \in G(n,m)$ with constants
  $\rho > 0$, $\varepsilon \in (0,1/2)$ and $\delta \in (0,1)$. Let
  $(u_1,\ldots,u_m)$ be some basis of $U \in G(n,m)$, such that $|u_i - v_i| \le
  \vartheta \rho$ for some $\vartheta \in (0,\tfrac{1}{\sqrt{2}}-\tfrac 14)$ and
  for each $i = 1,\ldots,m$. Furthermore, let us assume that
  \begin{equation}
    \label{cond:eps-del}
    C_3 (\varepsilon + C_2 \delta) < 1/2 \,.
  \end{equation}
  Then there exists a constant $C_4 =
  C_4(m,\varepsilon,\delta)$ such that
  \begin{displaymath}
    \dgras(U,V) \le C_4 \vartheta \,.
  \end{displaymath}
\end{proposition}

\begin{proof}
  Set $e_i := v_i / \rho$ and let $(\hat{e}_1,\ldots,\hat{e}_m)$ be the
  orthonormal basis of $V$ arising from $(e_1,\ldots,e_m)$ by virtue of
  Proposition~\ref{prop:gs-red}.  Set $f_i := u_i / \rho$.
  \begin{align*}
    |Q_U (\hat{e}_i)| &\le |Q_U (\hat{e}_i - e_i)| + |Q_U (e_i)|
    \le |\hat{e}_i - e_i| \dgras(U,V) + |e_i - f_i| \\
    &\le |\hat{e}_i - e_i| \dgras(U,V) + \vartheta \,.
  \end{align*}
  From Proposition~\ref{prop:gs-red} we have $|\hat{e}_i - e_i| \le \varepsilon
  + C_2 \delta$, so

  \begin{displaymath}
    |Q_U (\hat{e}_i)| 
    \le (\varepsilon + C_2 \delta) \dgras(U,V) + \vartheta
    \le 2(\epsilon+C_2\delta)+\vartheta
    \overset{\eqref{cond:eps-del}}{<}
    \frac 14 +\vartheta <\frac{1}{\sqrt{2}}\,,
  \end{displaymath}
  since $C_3(m)\ge 4$ for all $m\in\N$; see the definition of $C_3(m)$ at the
  end of the proof of Proposition \ref{prop:dist-ang}. Hence
  Proposition~\ref{prop:dist-ang} is applicable to the orthonormal basis
  $(\hat{e}_1, \ldots,\hat{e}_m)$ of $V$, and we conclude
  \begin{displaymath}
    \dgras(U,V) \le C_3 (\varepsilon + C_2 \delta) \dgras(U,V) + C_3 \vartheta
  \end{displaymath}
  \begin{displaymath}
    \text{hence} \qquad
    (1 - C_3 (\varepsilon + C_2 \delta)) \dgras(U,V) \le C_3 \vartheta \,.
  \end{displaymath}
  Since we assumed \eqref{cond:eps-del} we can divide both sides by $1 - C_3
  (\varepsilon + C_2 \delta)$ reaching the estimate
  \begin{displaymath}
    \dgras(U,V) \le \frac{C_3}{1 - C_3 (\varepsilon + C_2 \delta)} \vartheta \,.
  \end{displaymath}
  Finally we set
  \begin{displaymath}
    C_4 = C_4(m,\varepsilon,\delta) 
    := \frac{C_3(m)}{1 - C_3(m) (\varepsilon + C_2(m) \delta)} \,. \qedhere
  \end{displaymath}
\end{proof}

\subsection{Angles and intersections of tubes}

The results of this subsection are taken from our earlier work
\cite{svdm-tp1}. We are concerned with the intersection of two tubes whose
$m$-dimensional `axes' form a small angle, i.e. with the set
\begin{equation}
  \label{SHH} S(H_1,H_2):= \{y\in \R^n\, \colon \dist (y,H_i)\le
  1\quad \mbox{for $i=1,2$}\}, 
\end{equation}
where $H_1\not= H_2\in G(n,m)$ are such that $\pi_{H_1}$ restricted to $H_2$ is
bijective. Since the set $\{y\in \R^n\, \colon \dist (y,H_i)\le 1\}$ is convex, closed
and centrally symmetric\footnote{The term \emph{central symmetry} is used here
  for central symmetry with respect to $0$ in $\R^n$.} for each $i=1,2$, we
immediately obtain the following:

\begin{lemma}
  \label{cccs} $S(H_1,H_2)$ is a convex, closed and centrally symmetric set in
  $\bbbr^n$; $\pi_{H_1}(S(H_1,H_2))$ is a convex, closed and centrally symmetric
  set in $H_1\cong \R^m$.
\end{lemma}

For the global tangent-point curvature $\GTP$, the next lemma and its corollary
provide a key tool in bootstrap estimates in Section~\ref{bootstrap}.

\begin{lemma}
  \label{proj-slab}
  There exist constants $1> \eps_1=\eps_1(m)>0$ and $c_2(m)<\infty$ with the
  following property. If $H_1,H_2\in G(n,m)$ satisfy $0<\ang(H_1,H_2) = \alpha <
  \eps_1$, then there exists an $(m-1)$-dimensional subspace $W\subset H_1$ such
  that
  \[
  \pi_{H_1}\bigl(S(H_1,H_2)\bigr) \subset \{y\in H_1\colon \dist
  (y,W)\le 5c_2/\alpha\}\, .
  \]
\end{lemma}
For the proof, we refer to \cite[Lemma~2.6]{svdm-tp1}. It is an instructive
elementary exercise in classical geometry to see why this lemma is true for
$m=2$ and $n=3$.

\smallskip

The next lemma is now practically obvious.

\begin{lemma} \label{strip-ball} Suppose that $H\in G(n,m)$ and
  a set $S'\subset H$ is contained in $\{y\in H\colon \dist(y,W)\le
  d\}$ for some $d>0$, where $W$ is an $(m-1)$-dimensional subspace
  of $H$. Then
  \[
  \H^m\bigl(S'\cap \Ball^n(a,s) \bigr) \le 2^m s^{m-1}d\,
  \]
  for each $a\in H$ and each $s>0$.
\end{lemma}

\begin{proof}
  Writing each $y\in S'\cap \Ball^n(a,s)$ as $y= \pi_W(y) + (y-\pi_W(y))$,
  one sees that $S'\cap \Ball^n(a,s)$ is contained in a rectangular box with
  $(m-1)$ edges parallel to $W$ and of length $2s$ and the remaining edge
  perpendicular to $W$ and of length $2d$.
\end{proof}

  \reallyinvisible{}{
    {\tt\xx\xx I have commented the former Lemma 2.14 and its proof
      out, since
      we may not apply this in the slicing section... The material is still
      in the file\xx\xx} 
  }

%

\subsection{The voluminous simplices}

Several energy estimates for the global Menger curvature are based on
considerations of simplices that are roughly regular, which means that they have
all edges $\approx d$ and volume $\approx d^{m+1}$. Here are the necessary
definitions, making this vague description precise.

\begin{definition}
  Let $T = \conv(x_0,\ldots,x_{m+1})$ be an $(m+1)$-dimensional simplex in
  $\R^n$. For each $j = 0,\ldots,m+1$ we define the faces
  $\face_j(T)$, the heights $\height_j(T)$ and the minimal
  height  $\hmin(T)$ by
  \begin{align*}
    \face_j(T) &= \conv(x_0,\ldots,x_{j-1},x_{j+1},\ldots,x_{m+1}) \,, \\
    \height_j(T) &= \dist(x_j, \aff\{x_0,\ldots,x_{j-1},x_{j+1},\ldots,x_{m+1}\}) \\
    \text{and} \quad
    \hmin(T) &= \min\{ \height_i(T) : i = 0,1,\ldots,m+1 \} \,,
  \end{align*}
  where
  $\aff\{p_0,\ldots,p_{N}\}$ denotes the  (at most $N$-dimensional)
  affine plane spanned by 
$N+1$ the points 
$p_0,\ldots,p_N\in\R^n.$
\end{definition}
Note that for any $(m+1)$-dimensional simplex $T$ the volume
is given by
\begin{equation}\label{volume_simplex}
  \H^{m+1}(T)=\frac{1}{m+1}\height_i(T)\H^m(\face_i(T))\quad\textnormal{for any $i\in\{0,\ldots,m+1\}.$}
\end{equation}
The faces $\face_i(T)$ are lower-dimensional simplices themselves,
so that a simple inductive argument yields the estimate
\begin{equation}\label{volume_simplex_estimate}
  \H^{m+1}(T)\ge\frac{1}{(m+1)!}\height_{\textnormal{min}}(T)^{m+1}.
\end{equation}

\begin{definition}
  \label{def:voluminous}
  Fix some $\eta \in [0,1]$ and $d > 0$. Let $T = \conv(x_0,\ldots,x_{m+1})$ be
  an $(m+1)$-dimensional simplex in $\R^n$. We say that $T$ is {\it
    $(\eta,d)$-voluminous} and write $T \in \V(\eta,d)$ if the following
  conditions\footnote{A similar class of \emph{$1$-separated simplices\/} has been considered by Lerman and  Whitehouse in \cite[Section 3.1]{LW08b}} are satisfied
  \begin{displaymath}
    \diam(T) \le d
    \quad \text{and} \quad
    \hmin(T) \ge \eta d \,.
  \end{displaymath}
\end{definition}

\begin{proposition}
  \label{prop:perturbed}
  Let $T = \conv(x_0,\ldots,x_{m+1})$ be an $(\eta,d)$-voluminous simplex in
  $\R^n$ and set $\alpha = \frac{1}{8} \eta^2$. Let $\bar{x}_0 \in \R^n$ be such
  that $|x_0 - \bar{x}_0| \le \alpha d$ and set $\bar{T} =
  \conv(\bar{x}_0,x_1,\ldots,x_{m+1})$. Then
  \begin{displaymath}
    \diam(\bar{T}) \le \tfrac 98 d
    \quad\text{and}\quad
    \hmin(\bar{T}) \ge \tfrac 12 \eta d = \left( \tfrac 49 \eta \right) \left( \tfrac 98 d \right) \,.
  \end{displaymath}
  Thus, $\bar{T} \in \V\big(\tfrac 49 \eta, \tfrac 98 d\big)$\,.
\end{proposition}

\begin{proof}
  First we estimate the height $\height_0(\bar{T})$. Because $|x_0 - \bar{x}_0|
  \le \alpha d$ and $\eta \in [0,1]$ we have
  \begin{equation}
    \label{est:0height}
    \height_0(\bar{T}) \ge \height_0(T) - \alpha d 
    \ge (\eta - \alpha) d > \frac 12 \eta d \,.
  \end{equation}

  Fix two indices $i_1,i_2 \in \{1,2,\ldots,m+1\}$ such that $i_1 \ne i_2$. We
  shall estimate the height $\height_{i_1}(\bar{T})$. Without loss of generality
  we can assume that $x_{i_2}$ is placed at the origin. Furthermore, permuting
  the vertices of $T$ we can assume that $i_1 = 1$ and $i_2 = 2$. We need to
  estimate $\height_1(\bar{T})$. Set
  \begin{align*}
    P = \lin\{ x_0 - x_2, x_3 - x_2, \ldots, x_{m+1} - x_2 \}
    =\lin\{x_0,x_3,\ldots,x_{m+1}\}\\
    \bar{P} = \lin\{ \bar{x}_0 - x_2, x_3 - x_2, 
    \ldots, x_{m+1} - x_2 \}=\lin\{\bar{x}_0,x_3,\ldots,x_{m+1}\} \,.
  \end{align*}
  Now we can write
  \begin{align}
    \notag
    \height_1(\bar{T}) 
    &= \dist(x_1, \bar{P}) = |Q_{\bar{P}}(x_1)| \\
    &= |Q_P(x_1) - (Q_P(x_1) - Q_{\bar{P}}(x_1))| \notag \\
    &\ge |Q_P(x_1)| - |Q_P(x_1) - Q_{\bar{P}}(x_1)| 
    \label{est:jheight} \\
    &\ge \eta d - \|Q_P - Q_{\bar{P}}\| |x_1|\notag\\
    & \ge  (\eta - \dgras(P,\bar{P})) d \notag \,,
  \end{align}
  so all we need to do is to estimate $\dgras(P,\bar{P})$ from above unless
  $\ang(P,\bar{P})=0,$ in which case we are done anyway.

  For that purpose let $y_0:=\pi_{P\cap\bar{P}}(x_0)$ be the closest point to
  $x_0$ in the $(m-1)$-dimensional subspace $P\cap\bar{P}$. (Recall that
  $x_2=0$.) Set
  $$
  v_1:=\frac{x_0-y_0}{|x_0-y_0|}\in (P\cap\bar{P})^\perp,
  $$
  and choose an orthonormal basis $(v_2,\ldots,v_m)$ of $P\cap\bar{P}$. Since
  $y_0\in P\cap\bar{P}\subset\aff\{x_1,x_2,\ldots,x_{m+1}\}$ one has
  $$
  |x_0-y_0|\ge\dist(x_0,\aff\{x_1,x_2,\ldots,x_{m+1}\})\ge
  \height_{\textnormal{min}}(T)\ge\eta d,
  $$
  so that
  \begin{equation}
    \label{est:v1-dist}
    Q_{\bar{P}}(v_1) = \frac{Q_{\bar{P}}(x_0-y_0)}{|x_0-y_0|}
    = \frac{Q_{\bar{P}}(x_0)}{|x_0-y_0|} 
    = \frac{\dist(x_0,\bar{P})}{|x_0-y_0|}
    \le \frac{|x_0-\bar{x}_0|}{\eta d}
    \le \frac{\alpha}{\eta} \,.
  \end{equation}

   Choose any vector $\bar{v}_1 \in \bar{P}$ such that
  $(\bar{v}_1,v_2,\ldots,v_m)$ forms an orthonormal basis of $\bar{P}$.  Note
  that $\pi_{\bar{P}}(v_1)$ is orthogonal to $v_j$ for each $j =
  2,\ldots,m$. Indeed, if $j \in \{2,\ldots,m\}$, then we have
  \begin{align*}
    \langle \pi_{\bar{P}}(v_1), v_j \rangle 
    &= \Bigl\langle \sum_{i=2}^m
    \underbrace{\langle v_1, v_i \rangle}_{=0} v_i, v_j \Bigr\rangle +
    \big\langle \langle v_1, \bar{v}_1 \rangle \underbrace{\bar{v}_1, v_j \big\rangle}_{=0}=0. 
  \end{align*}
  Hence, for
  \begin{displaymath}
    w = \frac{\pi_{\bar{P}}(v_1)}{|\pi_{\bar{P}}(v_1)|} \,,
  \end{displaymath}
  we have $\bar{P} = \lin\{w,v_2,\ldots,v_m\}$ and $(w,v_2,\ldots,v_m)$ is also
  an orthonormal basis of $\bar{P}$. Moreover
  \begin{displaymath}
    |w - v_1| \le |w - \pi_{\bar{P}}(v_1)| + |\pi_{\bar{P}}(v_1) - v_1|
    = (1 - |\pi_{\bar{P}}(v_1)|) + |Q_{\bar{P}}(v_1)| \,.
  \end{displaymath}
  Using \eqref{est:v1-dist} we obtain $(1 - |\pi_{\bar{P}}(v_1)|) \le
  \alpha/\eta$, hence
  \begin{equation}
    \label{est:w-v1}
    |w - v_1| \le 2 \frac{\alpha}{\eta} \,.
  \end{equation}
  Let $h \in \Sphere^{n-1}$ be any unit vector in $\R^n$. We calculate
  \begin{align*}
    |\pi_P(h) - \pi_{\bar{P}}(h)| 
    &= \left| \sum_{j=2}^m \langle h, v_j \rangle v_j + \langle h, v_1 \rangle v_1 
      - \sum_{j=2}^m \langle h, v_j \rangle v_j - \langle h, w \rangle w \right| \\
    &\le |\langle h, (v_1 - w) \rangle v_1| + |\langle h, w \rangle (v_1 - w)|
    \le 2 |v_1 - w| \le 4 \frac{\alpha}{\eta} \,.
  \end{align*}
  This gives us the bound $\dgras(P,\bar{P}) \le 4\frac{\alpha}{\eta}$. Plugging
  this into \eqref{est:jheight} and recalling that $\alpha = \frac{1}{8} \eta^2$
  we get
  \begin{displaymath}
    \height_{i_1}(\bar{T}) = \height_1(\bar{T}) 
    \ge \big(\eta - 4\tfrac{\alpha}{\eta}\big) d 
    = \frac 12 \eta d \,.
  \end{displaymath}
  Since the index $i_1$ was chosen arbitrarily from the set $\{1,\ldots,m+1\}$,
  together with \eqref{est:0height} we obtain
  \begin{displaymath}
    \hmin(\bar{T}) \ge \frac 12 \eta d \,,
  \end{displaymath}
  which ends the proof.
\end{proof}

\subsection{Other auxiliary results}

The following theorem due to Haj\l{}asz gives a characterization of the Sobolev
space $W^{1,p}$ and is now widely used in analysis on metric spaces. We shall
rely on this result in Section~3.4.

\begin{theorem}[\textbf{Haj\l{}asz}\protect{\cite[Theorem
    1]{Haj96}}]\label{hajlasz-ptwise} Let $\Omega$ be a ball in $\R^m$ and
  $1<p<\infty$. Then a function $f\in L^p(\Omega)$ belongs to $W^{1,p}(\Omega)$
  if and only if there exists a function
  $g\in L^p(\Omega)$ such that
  \begin{equation}
    \label{ptwise}  |f(x)-f(y)|\le |x-y|\bigl(g(x)+g(y)\bigr)\, .
  \end{equation}
\end{theorem}

In fact, Haj\l{}asz shows that if $f\in W^{1,p}$, then \eqref{ptwise} holds for
$g$ equal to a constant multiple of the Hardy--Littlewood maximal function $M
(|D f|)$ of $|D f|$ defined as
$$
Mh(x):=\sup_{r>0}\stint_{\Ball^m(x,r)}h(y)\,dy.
$$
Conversely,
\[
\|f\|_{W^{1,p}}\approx \|f\|_{L^p} + \inf_g \|g\|_{L^p}\, ,
\]
where the infimum is taken over all $g$ for which \eqref{ptwise} holds. This
follows from the proof of Theorem 1 in \cite[p. 405]{Haj96}.

Recall that $\beta$ and $\theta$ numbers were defined by \eqref{def:beta}
and \eqref{def:theta}.

\begin{definition}[cf. \cite{davidkenigtoro}, Definition~1.3]
  \label{def:rfvc}
  We say that a compact set $\Sigma \subset \R^n$ is \emph{Reif\-en\-berg-flat}
  (of dimension $m$) \emph{with vanishing constant} if
  \begin{displaymath}
    \lim_{r \to 0} \sup_{x \in \Sigma} 
    \theta_\Sigma(x,r) = 0 \,.
  \end{displaymath}
\end{definition}

The following proposition was proved by David, Kenig and Toro. We will rely on
it in Section~3.1.
\begin{proposition}[cf. \cite{davidkenigtoro}, Proposition~9.1]
  \label{prop:dkt-reg}
  Let $\kappa \in (0,1)$ be given. Suppose $\Sigma$ is an $m$-dimensional
  compact Reifenberg-flat set with vanishing constant in $\R^n$ and that there
  is a constant $C_\Sigma$ such that
  \begin{displaymath}
    \beta_\Sigma(x,r) \le C_\Sigma r^{\kappa} \quad \text{for each $x \in \Sigma$ and $r \le 1$.}
  \end{displaymath}
  Then $\Sigma$ is an $m$-dimensional $C^{1,\kappa}$-submanifold of $\R^n$
  without boundary\footnote{Although boundaries of manifolds are not explicitly
    excluded in the statement of \cite[Proposition 9.1]{davidkenigtoro} it
    becomes evident from the proof that no boundaries are present; see in
    particular \cite[p. 433]{davidkenigtoro}.}.
\end{proposition}

\section{Towards the $W^{2,p}$ estimates for graphs}

\label{sec:3}

In this section we prove the harder part of the main result, i.e. the
implications (2)~$\Rightarrow$~(1) and (3)~$\Rightarrow$~(1). We follow the
scheme sketched in the introduction. Each of the four steps is presented in a
separate subsection.

\subsection{The decay of $\beta$ numbers and initial $C^{1,\kappa}$ estimates}

In this subsection we prove the following two results.

\begin{proposition}
  \label{prop:beta-est}
  Let $\Sigma \subset \R^n$ be an $m$-fine set, i.e. $
  \Sigma\in\F(m)$, such that
  $$\|\GMC\|_{L^{p}(\Sigma,\H^m)} \le E^{1/p}$$
  for some $E<\infty $ and some $p >  m$.
  Then, the inequality
  \[
  \beta_{\Sigma} (x,r)\le C\left(
    \frac{E}{A_\Sigma}\right)^{\kappa_1/(p-m)} r^{\kappa_1}\, , \qquad \kappa_1:=\frac{p-m}{p(m+1)+2m}\, ,
  \]
  holds for all $r\in (0, \diam\Sigma]$ and all $x\in\Sigma$. The constant $C$ depends on $m,p$ only.
\end{proposition}

\begin{proposition}
  \label{prop:beta-est-tp}
  Let $\Sigma\in\F (m)$ be an $m$-fine set such that
  \[
  \|\GTP\|_{L^p(\Sigma, \H^m)}\le E^{1/p} 
  \]
  for some map $H\colon \Sigma\to G(n,m)$, a constant  $E<\infty$ and some $p>m$. Then, the inequality
  \[
  \beta_{\Sigma} (x,r)\le C\left(
    \frac{E}{A_\Sigma}\right)^{\kappa_2/(p-m)} r^{\kappa_2}\, , \qquad \kappa_2:=\frac{p-m}{p+m}\, ,
  \]
  holds for all $r\in (0, \diam\Sigma]$ and all $x\in\Sigma$. The constant $C$
  is an absolute constant.
\end{proposition}

The argument is pretty similar in either case but it will be convenient to give
two separate proofs.

For the proof of Proposition~\ref{prop:beta-est} we mimic -- up to some
technical changes -- the proof of \cite[Corollary~2.4]{slawek-phd}. First we
prove a lemma which is an analogue of\cite[Proposition~2.3]{slawek-phd}.

\begin{lemma}
  \label{lem:eta-d-balance}
  Let $\Sigma \subset \R^n$ be an $m$-fine set, and let $x_0,x_1,\ldots,
  x_{m+1}$ be arbitrary points of $\Sigma$. Assume that $T =
  \conv(x_0,\ldots,x_{m+1})$ is $(\eta,d)$-voluminous for some $\eta \in (0,1)$
  and some $d \in (0,\infty)$. Furthermore, assume that
  $\|\GMC\|_{L^{p}(\Sigma,\H^m)} \le E^{1/p}$ for some $E < \infty$ and some
  $p>m$. Then there exists a constant $C=C(m,p)$ depending only on $m$ and $p$,
  such that
  \begin{displaymath}
    E \ge C A_{\Sigma} d^{m-p} \eta^{p(m+1) + 2m} \,.
  \end{displaymath}
  Equivalently,
  \begin{displaymath}
    \eta \le C' \left( \frac E{A_{\Sigma}} \right)^{\kappa_1/(p-m)} d^{\kappa_1}\,,
  \end{displaymath}
  where $C' = C'(m,p)$ and
  \[
  \kappa_1=\frac{p-m}{p(m+1)+2m}\, .
  \]
\end{lemma}

\begin{proof}
  Set $\alpha = \frac{1}{8} 
  \eta^2$. By Proposition~\ref{prop:perturbed},
  each $(m+1)$-simplex 
  $$
  \bar{T} = \conv(\bar{x}_0,x_1,\ldots,x_{m+1})
  $$ 
  satisfying $|x_0 - \bar{x}_0| \le \alpha d$ is $(\frac 49 \eta, \frac 98
  d)$-voluminous. Thus, for any such $\bar{T}$ we have according to
  \eqref{volume_simplex_estimate}
  \begin{equation}
    \label{est:DCbound}
    \DC(\bar{T}) \ge \frac{\big( \frac 49 \eta \big)^{m+1}}{(m+1)! \frac 98 d} 
    = C \frac{\eta^{m+1}}{d} \,,
  \end{equation}
  where $C = C(m) = (\frac 49)^{m+1} \frac 8{9(m+1)!}$. Using
  \eqref{est:DCbound} we obtain
  \begin{align*}
    E &\ge \|\GMC\|^p_{L^{p}(\Sigma,\H^m)}\\[4pt]
    &\ge \int_{\Sigma \cap \Ball(x_0,\alpha d)} \GMC(x)^p\       d\H^m(x) \\[4pt]     
    &\ge \Big(C \frac{\eta^{m+1}}{d}\Big)^p \H^m(\Sigma \cap \Ball(x_0,\alpha d))\\[4pt]        
    &\ge C^p  \left(\tfrac{1}{8}\right)^m A_{\Sigma} d^{m - p} \eta^{p(m+1) + 2m} \,.
  \end{align*}
  This completes the proof of the lemma.
\end{proof}

\medskip

We are now ready to give the {\it Proof of Proposition~\ref{prop:beta-est}}.

Fix some point $x \in \Sigma$ and a radius $r \in (0,\diam(\Sigma)]$. Let $T =
\conv(x_0, \ldots, x_{m+1})$ be an $(m+1)$-simplex such that $x_i \in \Sigma
\cap \CBall(x,r)$ for $i = 0,1,\ldots,m+1$ and such that $T$ has maximal
$\H^{m+1}$-measure among all simplices with vertices in $\Sigma \cap
\CBall(x,r)$, i.e.
\begin{displaymath}
  \H^{m+1}(T) = \max\{
  \H^{m+1}(\conv(x_0',\ldots,x_{m+1}')) : x_i' \in \Sigma \cap \CBall(x,r)
  \} \,.
\end{displaymath}
The existence of $T$ follows from the fact that the set $\Sigma \cap
\CBall(x,r)$ is compact and from the fact that the function $T \mapsto
\H^{m+1}(T)$ is continuous with respect to $x_0$, \ldots, $x_{m+1}$; see, e.g.,
formula \eqref{volume_simplex}.

Renumbering the vertices of $T$ we can assume that $\hmin(T) =
\height_{m+1}(T)$. Thus, according to \eqref{volume_simplex} the largest
$m$-face of $T$ is $\conv(x_0,\ldots,x_m)$ . Let $H = \lin\{ x_1-x_0, \ldots,
x_m-x_0 \}$, so that $x_0 + H$ contains the largest $m$-face of $T$. Note that
the distance of any point $y \in \Sigma \cap \CBall(x,r)$ from the affine plane
$x_0 + H$ has to be less then or equal to $\hmin(T) = \dist(x_{m+1},x_0+H)$,
since if we could find a point $y \in \Sigma \cap \CBall(x,r)$ with
$\dist(y,x_0+H) > \hmin(T)$, then the simplex $\conv(x_0,\ldots,x_m,y)$ would
have larger $\H^{m+1}$-measure than $T$ but this is impossible due to the choice
of $T$.

Since $x \in \Sigma \cap \CBall(x,r)$, we know that $\dist(x,x_0+H) \le
\hmin(T)$. Thus, we obtain for all $y \in \Sigma \cap \CBall(x,r) $
\begin{equation}
  \dist(y,x+H) \le \dist(y,x_0+H)+\dist(x,x_0+H)\le 2 \hmin(T).
\end{equation}
Hence
\begin{equation}
  \beta_\Sigma(x,r) \le \frac{2\hmin(T)}{r} \,.
  \label{est:beta-hmin}
\end{equation}
Now we only need to estimate $\hmin(T) = \height_{m+1}(T)$ from above.  
Of~course $T$ is $(\eta,2r)$-voluminous with $\eta = {\hmin(T)}/({2r})$.
Lemma~\ref{lem:eta-d-balance} implies that
\begin{displaymath}
  \beta_\Sigma(x,r) \le \frac{2\hmin(T)}{r} =
  4\eta
  \le
  C \left( \frac E{A_{\Sigma}} \right)^{\kappa_1/(p-m)} r^{\kappa_1} \,,
\end{displaymath}
which ends the proof of the proposition.
\qed

\medskip

Now we come to the {\it  Proof of Proposition~\ref{prop:beta-est-tp}.}

Fix $x\in \Sigma$ and $r\in (0,\diam\Sigma]$. We know by definition of the
$\beta$-numbers that \mbox{$\beta\equiv\beta_\Sigma(x,r)\le{}1$}. We also know
that for any $z\in \Sigma\cap\Ball(x,\beta r/2)$ that
$$
\sup_{\Sigma\cap\Ball(x,r)}\dist(\cdot,x+H_z)\ge \beta_\Sigma(x,r)r,
$$ 
where $H_z\in G(n,m)$ denotes the image of $z$ under the mapping $H:\Sigma\to
G(n,m)$. Furthermore, for any $\epsilon>0$ we can find a point
$y_\epsilon\in\Sigma\cap\Ball(x,r)$ such that
$$
\dist(y_\epsilon,x+H_z)\ge\sup_{\Sigma\cap\Ball(x,r)}\dist(\cdot,x+H_z)-\epsilon\ge\beta_\Sigma(x,r)r-\epsilon.
$$
On the other hand, we have  by $|y_\epsilon-z|\le|y_\epsilon -x|+|x-z|\le
\tfrac 32 r$
$$
\dist(y_\epsilon,z+H_z)\le\frac 12 \GTP(z)|y_\epsilon-z|^2
\le  \GTP(z) \frac 98 r^2
$$ 
so that  we obtain
\begin{eqnarray*}
  \frac 98 r^2\GTP(z) & \ge & \dist(y_\epsilon,z+H_z)\\
  & \ge & \dist(y_\epsilon,x+H_z)-|x-z|\\
  & \ge & \beta_\Sigma(x,r)r-\epsilon-\beta_\Sigma(x,r)r/2,
\end{eqnarray*}
which upon letting $\epsilon\to 0$ leads to
$$
\GTP(z)\ge \frac 49 \beta_\Sigma(x,r)/r.
$$
Estimating the energy as
\begin{align*}
  E &\ge \int_{\Sigma\cap\Ball(x,\beta r/2)}\GTP(z)^p\,d\H^m(z)\\
  &\ge \left(\frac 49 \right)^p(\beta_\Sigma(x,r))^pr^{-p}\H^m(\Sigma\cap\Ball(x,\beta r/2))
  \ge \left(\frac 49 \right)^p \left(\frac 12\right)^{m}  A_\Sigma r^{m-p}(\beta_\Sigma(x,r))^{p+m},
\end{align*}
which gives the desired estimate for $C =4 > \left(\frac 94 \right)^{p/(p+m)} 2^{m/(p+m)}$.
\qed

\bigskip

\begin{corollary}[\textbf{$C^{1,\kappa}$ estimates, first version}]
  Let $\Sigma \subset \R^n$ be an $m$-fine set and set
  $\mathcal{K}^{(1)}(\cdot):= \GMC[\Sigma](\cdot)$ and
  $\mathcal{K}^{(2)}(\cdot):=\GTP[\Sigma] (\cdot)$. If
  \[
  \int_\Sigma \GCi(z)^p\, d\H^m(z) \le E <\infty
  \]
  holds for $i=1$ or $i=2$. Then $\Sigma$ is an embedded closed manifold of
  class $C^{1,\kappa_i}$, where
  \[
  \kappa_1=\frac{p-m}{p(m+1)+2m}\, , \qquad \kappa_2=\frac{p-m}{p+m}\, .
  \]

  Moreover we can find a radius $R =
  R(n,m,p,A_{\Sigma},M_\Sigma,E,\diam\Sigma)$ 
  and a constant $K =
  K(n,m,p,A_{\Sigma},M_\Sigma,E,\diam\Sigma)$ such that
  for each $x \in \Sigma$ 
  there is a function 
  $$
  f_x:T_x\Sigma=:P
  \cong\R^m\to P^\perp\cong\R^{n-m}
  $$
  of class $C^{1,\kappa_i}$, such that $f_x(0)=0$ and
  $D f_x(0)=0$, and
  $$
  \Sigma\cap\Ball^n(x,R)=x+\Big(\graph f_x\cap\Ball^n(0,R)\Big),
  $$
  where $\graph f_x\subset P\times P^\perp=\R^n$ denotes the graph of $f_x$, and
  $$
  \|D f_x\|_{C^{0,\kappa_i}(\overline{\Ball^{m}(0,R)},
    \R^{(n-m)\times n})}\le K.
  $$
  \label{C1kappa-first}
\end{corollary}

\begin{proof}
  The first non-quantitative part follows from our estimates on the
  $\beta$-numbers in Proposition \ref{prop:beta-est} and \ref{prop:beta-est-tp}
  in combination with \cite[Proposition~9.1]{davidkenigtoro},
  cf. Proposition~\ref{prop:dkt-reg} of the previous section. However, direct
  arguments (as in \cite[Corollary~3.18]{slawek-phd} for the global Menger
  curvature $\GMC$, and in \cite[Section 5]{svdm-tp1} for the global
  tangent-point curvature $\GTP$), lead to the full statement of that corollary
  including the uniform estimates on the H\"older-norm of $Df_x$ and on the
  minimal size of the surfaces patches of $\Sigma$ that can be represented as
  the graph of $f_x$.  Let us give the main ideas here for the convenience of
  the reader.

  Assume without loss of generality that $x=0$ and write $\kappa:=\kappa_i$ for
  any $i\in\{1,2\}$ depending on the particular choice of integrand
  $\mathcal{K}^{(i)}$.  We know from Proposition \ref{prop:beta-est} or
  \ref{prop:beta-est-tp}, respectively, that there is a constant
  $C_1=C_1(A_\Sigma,E,m,p)$ such that
  \begin{equation}
    \label{BETA}
    \beta(r):=\beta_\Sigma(0,r)\le C_1 r^\kappa\quad\Foa r\in (0,\diam\Sigma].
  \end{equation}
  Since $\Sigma\in\F(m)$ we have
  \begin{equation}\label{THETA}
    \theta(r):=\theta_\Sigma(0,r)\le M_\Sigma C_1r^\kappa\quad\Foa
    r\in (0,\diam\Sigma].
  \end{equation}
  The Grassmannian $G(n,m)$ is compact, so we find for each $r \in
  (0,\diam\Sigma]$ an $m$-plane $H_{x}(r)\in G(n,m)$ such that
  $$
  \sup_{z\in\Sigma\cap\Ball(x,r)}\dist(z,H_x(r))=\beta(r)r.
  $$
  Taking an ortho-$(r/3)$-normal basis $(v_1(r),\ldots,v_m(r))$ of
  $H_x(r)$ for any such $r\in (0,\diam\Sigma]$ we find by
  \eqref{THETA} for each $i=1,\ldots,m$, some  point
  $z_i(r)\in\Sigma$ such that
  \begin{equation}\label{z-v-dist}
    |z_i(r)-v_i(r)|\le M_\Sigma C_1r^{\kappa+1};
  \end{equation}
  see Definition \ref{def:fine}. Now there is a radius
  $R_0=R_0(A_\Sigma,E,m,p,M_\Sigma)>0$ so small that  we have the inclusion
  $\Ball(v_i(r),M_\Sigma C_1r^{\kappa+1}) \subset \Ball(0,r/2)$ for each
  $r\in (0,R_0)$ and  each $i =1,\ldots,m$, which then implies by
  \eqref{BETA} that
  \begin{equation}\label{z-Hrhalbe}
    \dist(z_i(r),H_x(r/2))\le C_1r^{\kappa+1}\quad\Foa r\in (0,R_0).
  \end{equation}
  The orthogonal projections $u_i(r):=\pi_{H_x(r/2)}(v_i(r))$
  for $i=1,\ldots,m,$ satisfy due to \eqref{z-v-dist} and
  \eqref{z-Hrhalbe}
  $$
  |u_i(r)-v_i(r)|\le |v_i(r)-z_i(r)|+\dist(z_i(r),H_x(r/2))
  \le (M_\Sigma+1)C_1r^{\kappa+1}.
  $$
  Hence there is a smaller radius $0<R_1=R_1(A_\Sigma,E,m,p,M_\Sigma) \le R_0$
  such that for all $r\in (0,R_1)$ one has 
  \begin{equation}\label{new_small}
  C_1r^\kappa < (M_\Sigma+1)C_1r^\kappa<
  \frac{1}{\sqrt{2}}-\frac 14,
  \end{equation}
  so that Proposition~\ref{prop:red-ang} is applicable to the $(r/3,0,0)$-basis
  $(v_1(r),\ldots,v_m(r))$ of $V:=H_x(r)$ and the basis $(u_1(r),\ldots,u_m(r))$
  of $U:=H_x(r/2)$ with $\vartheta:=C_1r^\kappa.$ (Notice that condition
  \eqref{cond:eps-del} in Proposition \ref{prop:red-ang} is automatically
  satisfied since $\epsilon=\delta=0$ in the present situation.)  Consequently,
  \begin{equation}\label{angle-growth-estimate}
    \ang (H_x(r),H_x(r/2)) \le C_4C_1r^\kappa\quad\Foa r\in (0,R_1).
  \end{equation}
  Iterating this estimate, one can show that the sequence of $m$-planes
  $(H_x(r/2^N))$ is a Cauchy sequence in $G(n,m)$, hence converges as
  $N\to\infty$ to a limit $m$-plane, which must coincide with the already
  present tangent plane $T_0\Sigma$ at $x=0$, and the angle estimate
  \eqref{angle-growth-estimate} carries over to
  \begin{equation}
    \label{THang}
    \ang (T_x\Sigma,H_x(r))\le Cr^\kappa\quad\Foa r\in (0,R_1) \,.
  \end{equation}

  Let $y \in \Sigma$ be such that $|y-x| = r/2$ and set $w_i(r) =
  \pi_{H_y(r)}(v_i(r))$. We have $z_i(r) \in \Ball(y,r)$, so
  \begin{displaymath}
    \dist(z_i(r), H_y(r)) \le \beta_{\Sigma}(y,r) r \le C_1 r^{\kappa + 1} \,,
  \end{displaymath}
  \begin{displaymath}
    \text{hence} \quad
    |v_i(r) - w_i(r)| \le |v_i(r) - z_i(r)| + \dist(z_i(r), H_y(r))
    \le (M_\Sigma + 1)C_1 r^{\kappa + 1} \,.
  \end{displaymath}
  Applying once again Proposition~\ref{prop:red-ang} -- which is possible
  due to \eqref{new_small} -- we obtain the inequality
  \begin{displaymath}
  \ang(H_x(r), H_y(r)) \le  C_4(M_\Sigma+1)C_1 r^{\kappa} = \bar{C} |x-y|^{\kappa} \,.
  \end{displaymath}
This together with \eqref{THang} (which by symmetry also holds in $y$ replacing 
  $x$) leads to  the desired local estimate
  for the oscillation of tangent planes
  \begin{equation}
    \label{est:tanosc}
    \ang (T_x\Sigma, T_y\Sigma)\le C |x-y|^\kappa\quad\Foa |x-y|\le R_1/2,
  \end{equation}
  where $C=C(E,A_\Sigma,m,p,M_\Sigma)$ and $R_1=R_1(E,A_\Sigma,m,p,M_\Sigma)$ do {\it
    not} depend on the choice of $x,y\in\Sigma$.

  Next we shall find a radius $R_2=R_2(E,A_\Sigma,m,p,M_\Sigma)$ such that for
  each $x \in \Sigma$ the affine projection
  \begin{displaymath}
    \pi_x : \Sigma \cap \Ball(x,R_2) \to x + T_x\Sigma
  \end{displaymath}
  is injective. This will prove that $\Sigma \cap \Ball(x,R_2)$ coincides with
  a~graph of some function $f_x$, which is $C^{1,\kappa}$-smooth
  by~\eqref{est:tanosc}.

  Assume that there are two distinct points $y,z \in \Sigma \cap \Ball(x,R_1)$
  such that $\pi_x(y) = \pi_x(z)$. In other words $(y-z) \perp T_x\Sigma$. Since
  $y$ and $z$ are close to each other the vector $(y-z)$ should form a small
  angle with $T_z\Sigma$ but then $\ang(T_z\Sigma,T_x\Sigma)$ would be large and
  due to \eqref{est:tanosc} this can only happen if one of $y$ or $z$ is far
  from $x$. To~make this reasoning precise assume that $|x-y| \le |x-z|$ and set
  $H_x = H_x(|y-x|)$. Employing~\eqref{BETA} and~\eqref{THang} we get
  \begin{align*}
    |Q_{T_x\Sigma}(y-x)|
    &\le |Q_{H_x}(y-x)| + |Q_{T_x\Sigma}(y-x) - Q_{H_x}(y-x)| \\
    &\le \beta(x,|y-x|) |y-x| + \ang(T_x\Sigma,H_x) |y-x|
    \le C |y-x|^{1+\kappa} \le C |z-x|^{1+\kappa} \,,
  \end{align*}
  where $C$ depends only on $E$, $A_\Sigma$, $m$ and $p$. The same applies to
  $(z-x)$ so we also have
  \begin{displaymath}
    |Q_{T_x\Sigma}(z-x)| \le C |z-x|^{1+\kappa} \,.
  \end{displaymath}
  Next we estimate
  \begin{equation}
    \label{est:z-y}
    |z-y| = |Q_{T_x\Sigma}(z-y)|
    \le | Q_{T_x\Sigma}(z-x)| + |Q_{T_x\Sigma}(y-x)| 
    \le 2 C |z-x|^{1+\kappa} \,. 
  \end{equation}
  Setting $H_z = H_z(|y-z|)$ and repeating the same calculations we obtain
  \begin{displaymath}
    \dist(y-z,T_z\Sigma) = |Q_{T_z\Sigma}(y-z)| \le C |y-z|^{1+\kappa} \,.
  \end{displaymath}
  This gives
  \begin{align*}
    \ang(T_x\Sigma,T_z\Sigma) &= \| Q_{T_x\Sigma} - Q_{T_z\Sigma} \|
    \ge |Q_{T_x\Sigma}(z-y) - Q_{T_z\Sigma}(z-y)| |z-y|^{-1} \\
    &\ge \left( |z-y| - |Q_{T_z\Sigma}(z-y)| \right) |z-y|^{-1}
    \ge 1 - C |y-z|^{\kappa} \,.
  \end{align*}
  On the other hand by~\eqref{est:tanosc}
  \begin{math}
    \ang(T_x\Sigma,T_z\Sigma) \le C |x-z|^{\kappa} \,.
  \end{math}
  Hence, applying~\eqref{est:z-y} we obtain
  \begin{displaymath}
    C |x-z|^{\kappa} \ge 1 - \tilde{C} |y-z|^{\kappa} 
    \ge 1 - \bar{C} |x-z|^{\kappa+\kappa^2} \\
    \, \iff \, 
    |x-z| \ge \left(C + \bar{C}|x-z|^{\kappa^2} \right)^{-1/\kappa} \,.
  \end{displaymath}
  This shows that if $(y-z) \perp T_x\Sigma$ then the point $z$ has to be far
  from $x$. We set 
  \begin{displaymath}
    R_2 = \min\left(1, (C + \bar{C})^{-1/\kappa}\right)
  \end{displaymath}
  and this way we make sure that $\pi_x : \Sigma \cap \Ball(x,R_2) \to x +
  T_x\Sigma$ is injective for each $x \in \Sigma$, hence $\Sigma \cap
  \Ball(x,R_2)$ is a graph of some function $f_x : T_x\Sigma \cap \Ball(0,R_2)
  \to (T_x\Sigma)^{\perp}$.  

  The oscillation estimate~\eqref{est:tanosc} leads with standard arguments (as,
  e.g., presented in \cite[Section 5]{svdm-tp1}) to the desired uniform
  $C^{1,\kappa}$-estimates for $f_x$ on balls in $T_x\Sigma$ of radius $R_2$
  which depends on $E,A_\Sigma,p,m,M_\Sigma$, but not on the particular choice
  of the point $x$ on $\Sigma$.
\end{proof}

\begin{remark}
  \label{rem:Msigma}
  The statement of Corollary \ref{C1kappa-first} can a posteriori be sharpened:
  One can show that one can  make the constants $R$ and $K$ independent of
  $M_{\Sigma}$. This was carried out in detail in the first author's
  doctoral thesis; see \cite[Theorem 2.13]{slawek-phd}, so we will restrict
  to a brief sketch of the argument here. Assume as before that $x=0$ and
  notice that $\beta(r)=\beta(0,r) \to 0$ uniformly (independent of the point
  $x$ and also independent of $M_\Sigma$ according to \eqref{BETA}). Since at
  this stage we know that $\Sigma $ is a $C^{1,\kappa}$-submanifold of $\R^n$
  without boundary, it is clearly also admissible in the sense of
  \cite[Definition 2.9]{svdm-tp1}. In particular $\Sigma$ is locally flat around
  each point $y\in\Sigma$ -- it is actually close to the tangent $m$-plane
  $T_y\Sigma$ near $y$ -- and $\Sigma $ is nontrivially linked with sufficiently
  small $(n-m-1)$-spheres contained in the orthogonal complement of $T_y\Sigma.$
  Let $H_x(r)$ for $r\in (0,\diam\Sigma]$ be as in the proof of Corollary
  \ref{C1kappa-first} the optimal $m$-plane through $x=0$ such that
  \begin{equation}\label{CLOSE}
    \dist(y,x+H_x(r)) \le \beta(r)r
    \quad \Foa y \in \Sigma \cap \Ball(0,r).
  \end{equation}
  One can use now the uniform estimate \eqref{BETA} (not depending on
  $M_\Sigma$) to prove that there is a radius $R_3=R_3(E,A_\Sigma,m,p)$ such
  that the angle $\ang(T_0\Sigma,H_x(r))$ is for each $r \in (0,R_3)$ so small
  that, for any given $p\in H_x(r)\cap\Ball(0,R_3)$, one can deform the linking
  sphere in the orthogonal complement of $T_0\Sigma$ with a homotopy to a small
  sphere in $p+H_x(r)^\perp$ without ever hitting $\Sigma$. Because of the
  homotopy invariance of linking one finds also this new sphere nontrivially
  linked with $\Sigma$. This implies in particular by standard degree arguments
  the existence of a point $z\in\Sigma$ contained in the $(n-m)$-dimensional
  disk in $p+H_x(r)^\perp$ spanned by this new sphere; see,
  e.g.~\cite[Lemma~3.5]{svdm-tp1}. On the other hand by \eqref{CLOSE}
  $\Sigma\cap \Ball(0,r)$ is at most $\beta(r)r$ away from $H_x(r)$ which
  implies now that this point $z\in\Sigma$ must satisfy $|z-p|\le\beta(r)r$.
  This gives the uniform estimate $\theta(r)\le C\beta(r)$ for all  $r < R_3$
  and some absolute constant $C$.
\end{remark}

Now we know that the estimates in Corollary \ref{C1kappa-first} do not depend on
$M_\Sigma$. This constant may be replaced by an absolute one if we are only
working in small scales. In the next section we show that this can be further
sharpened: $R$ and $K$ depend in fact only on $m$, $p$ and $E$, but {\it not} on
the constant~$A_\Sigma$.

\subsection{Uniform Ahlfors regularity and its consequences}
\label{sec:tangent-planes}

In this section, we show that the $L^p$-norms of the global curvatures $\GMC$
and $\GTP$ control the length scale in which bending (or `hairs', narrow
tentacles, long thin tubes etc.)  can occur on $\Sigma$. In particular, there is
a number $R$ depending \emph{only} on $n,m,p$ and $E$, where $E$ is any constant
dominating $\|\GMC\|_{L^p}^p$ or $\|\GTP\|_{L^p}^p$, such that for all $x\in
\Sigma$ and all $r\le R$ the intersection $\Sigma\cap \Ball^n(x,r)$ is congruent
to $\graph f_x\cap \Ball^n(x,r)$, where $f_x\colon \R^m\to\R^{n-m}$ is a
$C^{1,\kappa_i}$ function (with small $C^1$ norm, if one wishes).  Note that $R$
does not at all depend on the shape or on other properties of $\Sigma$, just on
its energy value, i.e. on the $L^p$-norm of $\GMC$ or of $\GTP$.

By the results of the previous subsection, we already know that $\Sigma$ is an
embedded $C^1$ compact manifold without boundary. This is assumed throughout
this subsection.

The crucial tool needed to achieve such control over the shape of $\Sigma$ is
the following.

\begin{theorem}[\textbf{Uniform Ahlfors regularity}]
  \label{thm:UAR} For each $p>m$ there exists a constant 
  $C(n,m,p)$ with the following property. 
  If $\|\GMC\|_{L^p}$ or $\|\GMC\|_{L^p}$ is less than $E^{1/p}$
  for some $E<\infty$, then for every $x\in\Sigma$
  \begin{equation}
    \label{ineq:UAR}
    \H^m(\Sigma\cap \Ball^n(x,r))\ge \frac 12 \omega_mr^m \qquad\mbox{for all $0<r\le R_0$,}
  \end{equation}
  where
  $	
  R_0= C(n,m,p) E^{-1/(p-m)}
  $
  and $\omega_m=\H^m(\Ball^m(0,1)).$
\end{theorem}

The proof of Theorem~\ref{thm:UAR} is similar to the proof of Theorem~3.3 in
\cite{svdm-surfaces} where Menger curvature of surfaces in $R^3$ has been
investigated. This idea has been later reworked and extended in various settings
to the case of sets having codimension larger than 1.

Namely, one demonstrates that each $\Sigma$ with finite energy cannot penetrate
certain conical regions of $\R^n$ whose size \emph{depends solely on the
  energy}. The construction of those regions has algorithmic nature. Proceeding
iteratively, one constructs for each $x\in \Sigma$ an increasingly complicated
set $S$ which is centrally symmetric with respect to $x$ and its intersection
with each sphere $\partial \Ball^n(x,r)$ is equal to the union of two or four
spherical caps. The size of these caps is proportional to $r$ but their position
may change as $r$ grows from $0$ to the desired large value, referred to as
\emph{the stopping distance} $d_s(x)$.  The interior of $S$ contains no points
of $\Sigma$ but it contains numerous $(n-m-1)$-dimensional spheres which are
nontrivially linked with $\Sigma$. Due to this, for each $r$ below the stopping
distance, $\Sigma\cap \Ball^n(x,r)$ has large projections onto some planes in
$G(n,m)$. However, there are points of $\Sigma$ on $\partial S$, chosen so that
the global curvature $\GMC(x)$, or $\GTP(x)$, respectively, must be $\gtrsim
1/d_s(x)$.

To avoid entering into too many technical details of such a construction, we
shall quote almost verbatim two purely geometric lemmata from our previous work
that are independent of any choice of energy, and indicate how they are used in
the proof of Theorem~\ref{thm:UAR}.

\subsubsection{The case of global Menger curvature}

Recall the Definition~\ref{def:voluminous} of the class $\V(\eta, d)$ of
$(\eta,d)$-voluminous simplices. The following proposition comes from the
doctoral thesis of the first author, see \cite[Proposition 2.5]{slawek-phd}.

\begin{proposition}
  \label{prop:big-proj-fat-simp}
  Let $\delta \in (0,1)$ and $\Sigma$ be an embedded $C^1$ compact manifold
  without boundary. There exists a real number $\eta = \eta(\delta,m) \in (0,1)$
  such that for every point $x_0 \in \Sigma$ there is a stopping distance $d =
  d_s(x_0) > 0$, and an $(m+1)$-tuple of points $(x_1, x_2, \ldots, x_{m+1}) \in
  \Sigma^{m+1}$ such that
  \begin{displaymath}
    T = \mathrm{conv}\{ x_0, \ldots, x_{m+1}\} \in \V(\eta, d) \,.
  \end{displaymath}
  Moreover, for all $\rho \in (0, d)$ there exists an $m$-dimensional subspace
  $H = H(\rho) \in G(n,m)$ with the property
  \begin{equation}
    \label{eq:big-proj}
    (x_0 + H) \cap \Ball^n(x_0, \sqrt{1 - \delta^2} \rho)
    \subset \pi_{x_0 + H}(\Sigma \cap \Ball^n(x_0, \rho)) \,.
  \end{equation}
\end{proposition}

Fixing $\delta=\delta(m)\in (0,\sqrt{1 - 4^{-1/m}})$ small enough, we obtain
$\eta=\eta(m)$ depending on $m$ only.  This yields the following.
\begin{corollary}
  \label{cor:uahlreg}
  For any $x_0 \in \Sigma$ and any $\rho \le d_s(x_0)$ we have
  \begin{equation}
    \label{eq:uahlreg}
    \H^m(\Sigma \cap \Ball(x_0,\rho)) 
    \ge (1 - \delta^2)^{ m/2} \omega_m \rho^m
    \ge \frac 12 \omega_m \rho^m\,.
  \end{equation}
\end{corollary}

Moreover, we can provide a lower bound for all stopping distances. For this,
we~need an elementary consequence of the definition of voluminous simplices:
\begin{fact}
  \label{fact:reg-curv}
  If $T=\conv(x_0,\ldots,x_{m+1}) \in \V(\eta,d)$ then by
  \eqref{volume_simplex_estimate}
  \begin{equation}
    \label{est:curv}
    K(x_0,\ldots,x_{m+1}) 
    \ge \frac{(\eta d)^{m+1}}{(m+1)! (d)^{m+2}} 
    = \frac{\eta^{m+1}}{(m+1)!d} \,.
  \end{equation}
\end{fact}
\noindent
For $\eta=\eta(m)$ and $d=d_s(x_0)$ this yields
\[
\GMC(x_0)\ge K(x_0,\ldots,x_{m+1})\ge \frac{a(m)}{d_s(x_0)}
\]
for some constant $a(m)$ depending only on $m$. By
Proposition~\ref{prop:perturbed}, we know that for simplices $\bar{T}$ that
arise from $T$ by shifting $x_0$ by at most $\frac 18 \eta^2 d$ a similar
estimate holds, possibly with a slightly smaller $a(m)$ -- still, depending only
on $m$. Thus,
\begin{equation}
  \label{GMC-lowbound}
  \GMC(z) \ge \frac{a(m)}{d_s(x_0)}\,,
  \quad\Foa
  z\in \Sigma \cap \Ball^n(x_0,\eta^2d/8)\,.
\end{equation}
Using the assumption of Theorem~\ref{thm:UAR} we now estimate
\begin{align*} 
  E &\ge \int_{\Sigma \cap \Ball^n(x_0,\eta^2d/8)} \GMC(z)^p\, d\H^m(z) \\
  &\ge \H^m(\Sigma \cap \Ball^n(x_0,\eta^2d/8)) \left( \frac {a(m)}{d_s(x_0)} \right)^p
  \qquad\text{by \eqref{GMC-lowbound}} \\
  &\ge \frac 1{2\cdot 8^m}\omega_m \eta^{2m}d_s(x_0)^{m-p}a(m)^p
  \qquad \text{by Corollary~\ref{cor:uahlreg}.}
\end{align*}
Note that $\eta \in (0,1)$, so Corollary~\ref{cor:uahlreg} is indeed
applicable. Equivalently,
\[
d_s(x_0)^{p-m}\ge c / E
\]
for some $c$ depending only on $m$ and $p$. Upon taking the infimum
w.r.t. $x_0\in \Sigma$ (note that we use $p>m$ here!), we obtain
\[
d(\Sigma) := \inf_{x_0\in\Sigma} d_s(x_0) \ge \left(\frac c{E}\right)^{1/(p-m)} =:R_0
\]
An application of Corollary~\ref{cor:uahlreg} implies now Theorem~\ref{thm:UAR}
in the case of $\GMC$.

\subsubsection{The case of global tangent--point curvature}

As we have already mentioned in the introduction, the $L^p$ norm of the global
tangent-point curvature $\GTP[\Sigma]$ can be finite for at most one choice of a
continuous map $H\colon \Sigma\ni x\mapsto H(x)\in G(n,m)$. Thus, from now on we
suppose
\[
H\colon \Sigma\ni x\longmapsto T_x\Sigma\in G(n,m)\, ,
\]
since at this point we know already that $\Sigma$ is a $C^1$ submanifold of
$\R^n$ (without boundary). The general scheme of proof is similar to the case of
global Menger curvature. Some of the technical details are different and we
present them below.

\subsubsection*{High energy couples of points and large projections}

\label{sec:4.1}

The notion of a \emph{high energy couple} expresses in a quantitative way the
following rough idea: if there are two points $x,y\in \Sigma$ such that the
distance from $y$ to a substantial portion of the affine planes $z+T_z\Sigma$
(where $z$ is very close to $x$) is comparable to $|x-y|$, then a certain fixed
portion of the `energy', i.e. of the norm $\|\GTP\|_{L^p}$, comes \emph{only
  from a fixed neighbourhood of $x$}, of size comparable to $|x-y|$.

Recall that $Q_{T_z\Sigma}$ stands for the orthogonal projection onto
$(T_z\Sigma)^\perp$.

\begin{definition}[High energy couples]
  \label{couples}
  We say that $(x,y)\in \Sigma \times \Sigma$ is a {\rm
    $(\lambda,\alpha,d)$--high energy couple} if and only if the following two
  conditions are satisfied:
  \begin{enumerate}
  \item[{\rm (i)}] $d/2 \le |x-y|\le 2d$;
  \item[{\rm (ii)}] The set
    \[
    S(x,y;\alpha,d) := \left\{ z\in \Ball^n(x,\alpha^2 d) \cap \Sigma\colon
    |Q_{T_z\Sigma}(y-z)|\ge \alpha d \right\}
    \]
    satisfies
    \[
    \H^m(S(x,y;\alpha,d)) \ge \lambda \H^m(\Ball^m(0,\alpha^2 d)) =
    \lambda \omega_m \alpha^{2m}d^m\, .
    \]
  \end{enumerate}
\end{definition}
We shall be using this definition for fixed $0<\alpha,\lambda\ll 1$ depending
only on $n$ and $m$. Intuitively, high energy couples force the $L^p$-norm of
$\GTP$ to be large.

\begin{lemma}
  \label{1/R-est}
  If $(x,y)\in \Sigma \times \Sigma$ is a $(\lambda,\alpha,d)$--high energy
  couple with $\alpha < \frac 12$ and an arbitrary $\lambda\in (0,1]$, then
  \begin{equation}
    \GTP(z) > \frac \alpha {9d}
  \end{equation}
  for all $z\in S(x,y;\alpha,d)$.
\end{lemma}

\begin{proof}
  For $z\in S(x,y;\alpha,d)$ and $w\in \Ball^n(y,\alpha^2 d)$  we have
  \begin{eqnarray*}
    \dist(w,z+T_z\Sigma)=|Q_{T_z\Sigma}(w-z)| & = & |Q_{T_z\Sigma}(y-z) + Q_{T_z\Sigma}(w-y)| \\
    & \ge & \alpha d - |w-y| \qquad \mbox{by Definition~\ref{couples}~(ii)}
    \\
    &  > & \frac{\alpha d}{2} \qquad\text{as } \alpha < \tfrac 12 \,.
  \end{eqnarray*}
  Moreover, $|w-z|\le |x-y|+|x-z|+|w-y|< 2d + 2\alpha^2d < 3d$.  Thus, by the
  above computation,
  \begin{align*}
    \GTP(z) &= \sup_{w\in \Sigma}\frac{2\dist(w,z+T_z\Sigma)}{|w-z|^2} \\
    &\ge   \sup_{w\in \Sigma\cap \Ball^n(y,\alpha^2 r)}\frac{2\dist(w,z+T_z\Sigma)}{|w-z|^2} 
    > \frac{\alpha
      d}{(3d)^2} =  \frac \alpha {9d}\, .
  \end{align*}
  This completes the proof of the lemma.
\end{proof}

\reallyinvisible{}{ To describe the mechanism of finding the high energy couples,
  we introduce some notation first. For a plane $H\in G(n,m)$ and $\delta\in
  (0,1)$ we set
  \begin{eqnarray}
    C(\delta,H) &:=& \{ z\in \bbbr^n \colon |Q_H(z)|\ge \delta |z|\}\, , \label{CH}\\
    C_r(\delta,H) &:=& C(\delta,H) \cap \overline{\Ball^n(0,r)}\, .\label{CrH}
  \end{eqnarray}
  (These are closed `double cones' with `axis' equal to $H^\perp$.  Note that if
  $n>m+1$, then the interior of $C(\delta,H)$ and of $C_r(\delta,H)$ is
  connected.) We shall also use the intersections of cones with annuli, {\tt\xx
    NECESSARY??????, SEE OUR REVISED \cite{svdm-tp1} }
  \begin{equation}
    A_{R,r}(x,\delta,W):=x+\mathrm{int}\, \Bigl(C_R(\delta,W)\setminus
    \Ball^n(0,r)\Bigr)\, . \label{ann-cone}
  \end{equation} 
}
  
The key to Theorem~\ref{thm:UAR} in the case of $\GTP$ global curvature
is to observe that high energy couples and large projections coexist on the same
scale.

\begin{proposition}[Stopping distances and large projections]
  \label{mainlemma-tp}
  There exist constants $\eta=\eta(m), \delta=\delta(m), \lambda=\lambda(n,m)
  \in (0,\frac 19) $ which depend only on $n,m$, and have the following
  property.

  Assume that $\Sigma$ is an arbitrary embedded $C^1$ compact manifold without
  boundary. For every $x\in \Sigma$ there exist a number $d\equiv d_s(x)>0$ and
  a point $y\in \Sigma$ such that
  \begin{enumerate}
    \renewcommand{\labelenumi}{{\rm (\roman{enumi})}}
    
  \item $(x,y)$ is a $(\lambda,\eta,d)$--high energy couple;
    
  \item for each $r\in (0,d]$ there exists a plane $H(r)\in G(n,m)$
    such that
    \[
    \pi_{H(r)}(\Sigma\cap \Ball^n(x,r)) \ \supset\ H(r)\cap
    \Ball^n\bigl(\pi_{H(r)}(x),r\sqrt{1-\delta^2}\bigr)\, ,
    \]
    and therefore $$ \H^m(\Sigma\cap \Ball^n(x,r)) \ge
    (1-\delta^2)^{m/2}\omega_m r^m\ge \frac 12 \omega_m r^m$$ for all $0<r \le d_s(x)$.  
    \reallyinvisible{}{
      
    \item the plane $W=H(d)\in G(n,m)$ is such that
      $\Sigma\cap A_{d,d/2}(x,\delta,W)=\emptyset$.

    \item Each disk $D^{n-m}(z,r;W^\perp)=  z +\{v\in W^\perp\colon
      |v|\le r\}$ with $z\in x+W$, $|z-x|\le d\sqrt{1-\delta^2}$, and
      radius $r$ such that
      \begin{equation}
        \S^{n-m-1}(z,r;W^\perp):= z +\{v\in W^\perp\colon |v|= r\} \
        \subset \ A_{d,d/2}(x,\delta, W) \label{s-in-cone}
      \end{equation}
      contains at least one point of $\Sigma$.
    }
  \end{enumerate}
\end{proposition}

For the proof of this lemma (for a much wider class of $m$-dimensional sets than
just $C^1$ embedded compact manifolds) we refer the reader to \cite[Section
4]{svdm-tp1}.

\begin{lemma}\label{low-dE-bounds} If $\Sigma\subset \R^n$ is an embedded $C^1$
  compact manifold without boundary, $p>m$ and
  \[
  E\ge \int_\Sigma \GTP(x)^p \, d\H^m(x)\, ,
  \]
  then the stopping distances $d_s(x)$
  of Proposition \ref{mainlemma-tp}
  satisfy
  \begin{equation}
    \label{low-E-bound}
    d(\Sigma) =\inf_{x\in\Sigma} d_s(x) \ge \left(\frac c{E}\right)^{1/(p-m)} =:R_0
  \end{equation}
  where $c$ depends only on $n$, $m$ and $p$.
\end{lemma}

\begin{proof}
  Let $\lambda$ and $\eta$ be the constants of 
  Proposition~\ref{mainlemma-tp}.
  Use this proposition  to select a $(\lambda,\eta,d)$--high energy couple $(x,y)\in \Sigma\times\Sigma$. Let
  \[
  S:=S(x,y;\eta,d_s(x))
  \]
  be as in Definition~\ref{couples}~(ii). Applying~Lemma~\ref{1/R-est} we estimate
  \begin{eqnarray*} 
    E & \ge & \int_{S}      \GTP(z)^p\, d\H^m(z) \\
    &  > & \H^m(S)  \left(\frac \eta{9d_s(x)}\right)^p \qquad\mbox{by Lemma~\ref{1/R-est}} \\
    &\ge & \lambda\omega_m \eta^{2m+p}d_s(x)^{m-p}9^{-p}  
    \qquad\mbox{by Definition~\ref{couples}~(ii).}
  \end{eqnarray*}
  This implies
  \begin{displaymath}
    d_s(x)^{p-m}
    >   c/E
  \end{displaymath}
  for a constant $c$ depending only on $n$, $m$, $p$. As in the case of $\GMC$,
  upon taking the infimum of the left hand side w.r.t. $x\in \Sigma$, we
  conclude the proof of the lemma.
\end{proof}

Theorem~\ref{thm:UAR} in the case of $\GTP$ follows now immediately. By the lower bound  \eqref{low-E-bound} for stopping distances and Proposition~\ref{mainlemma-tp}~(ii), the inequality
\[
\H^m(\Sigma\cap \Ball(x,r)) \ge 
(1-\delta^2)^{m/2} \omega_mr^m \ge \frac 12 \omega_mr^m
\]
holds for each $x\in \Sigma$ and each $r\le R_0$, since $R_0\le d(\Sigma)\le d_s(x)$.

\subsubsection{An application: uniform size of $C^{1,\kappa}$-graph patches}

Now, returning to the proofs of Propositions~\ref{prop:beta-est} and~\ref{prop:beta-est-tp}, we see that for all radii
\[
r\le C(n,m,p) E^{-1/(p-m)}=R_0
\]
the estimate $\H^m(\Ball(x,r))\ge A_\Sigma\omega_mr^m$ can be replaced by \eqref{ineq:UAR}, i.e. used with $A_\Sigma=1/2$. Thus, for such radii the decay estimates in  Propositions~\ref{prop:beta-est} and~\ref{prop:beta-est-tp}, and the resulting $C^{1,\kappa}$-estimates do not depend on $A_\Sigma$ or $\diam \Sigma$ at all. An inspection of the argument leading to Corollary~\ref{C1kappa-first} gives the following sharpened version, with all estimates depending in a uniform way only on the energy.

\begin{corollary}[\textbf{$C^{1,\kappa}$ estimates, second version}]\label{C1kappa-second}
  Assume that $\Sigma \subset \R^n$ is an $m$-fine set
  and let $\mathcal{K}^{(1)}(\cdot):=
  \GMC[\Sigma](\cdot)$ and $\mathcal{K}^{(2)}(\cdot):=\GTP[\Sigma]
  (\cdot)$. If
  \[
  \int_\Sigma \GCi(z)^p\, d\H^m(z) \le E <\infty
  \]
  holds  for $i=1$ or $i=2$. Then $\Sigma$ is an
  embedded closed manifold of class $C^{1,\kappa_i}$, where
  \[
  \kappa_1=\frac{p-m}{p(m+1)+2m}\, , 
  \qquad \kappa_2=\frac{p-m}{p+m}\, .
  \]

  Moreover we can find a radius $R_1 =
  a(n, m,p)E^{-1/(p-m)}\le R_0$ and a constant $K_1 =
  K(n,m,p)$ such that
  for each $x \in \Sigma$
  there is a function
  $$
  f_x:T_x\Sigma=:P
  \cong\R^m\to P^\perp\cong\R^{n-m}
  $$
  of class $C^{1,\kappa_i}$, such that $f_x(0)=0$ and
  $D f_x(0)=0$, and
  $$
  \Sigma\cap\Ball^n(x,R_1)=x+\Big(\graph f_x\cap\Ball^n(0,R_1)\Big),
  $$
  where $\graph f_x\subset P\times P^\perp=\R^n$ denotes the
  graph of $f_x$, and
  $$
  \|D f_x\|_{C^{0,\kappa_i}(\overline{\Ball^n}(0,R_1),
    \R^{(n-m)\times n})}\le K_1 E^{\kappa_i/(p-m)}\,.
  $$
\end{corollary}

As for Corollary \ref{C1kappa-first} also here we do not enter into the details
of construction of the graph parametrizations $f_x$. These are described in
\cite[Section~5.4]{svdm-tp1} and in \cite[Section~3]{slawek-phd}.

\begin{remark}\label{rem:flat} Note that shrinking $a(n,m,p)$ 
  if necessary, we can always assume that
  \begin{multline*}
    |Df_x(z_1)-Df_x(z_2)|  \le  K_1 E^{\kappa_i/(p-m)}
    \cdot R_1^{\kappa_i}\\= K_1 a(n,m,p)^{\kappa_i} E^{\kappa_i/(p-m)} E^{-\kappa_i/(p-m)} 
    = K_1(m,p)\cdot a(n,m,p)^{\kappa_i} <\eps_0
  \end{multline*}
  for an arbitrary small $\eps_0=\eps_0(m)>0$ that has been a priori fixed.
\end{remark}

\subsection{Bootstrap: optimal H\"older regularity for graphs}
\label{bootstrap}

In this subsection we assume that $\Sigma$ is a flat $m$-dimensional graph of
class $C^{1,\kappa_i}$, satisfying
\[
\int_\Sigma \GCi(z)^p\, d\H^m(z)<\infty\, 
\]
for $i=1$ or $i=2$, recall our notation from before: $\mathcal{K}^{(1)}:=\GMC$
and $\mathcal{K}^{(2)}:=\GTP.$ The goal is to show how to bootstrap the
H\"{o}lder exponent $\kappa_i$ to $\tau = 1- m/p$.

Relying on Corollary~\ref{C1kappa-second} and Remark~\ref{rem:flat}, without
loss of generality we can assume that
\[
\Sigma \cap \Ball^n(0,20 R) = \mathrm{Graph}\, 
f \cap \Ball^n(0,20 R)
\]
for a fixed number $R>0$, where
\[
f\colon P\cong \R^m \to P^\perp \cong \R^{n-m}
\]
is of class $C^{1,\kappa_i}$ and satisfies $D f(0)=0$,  $f(0)=0$,
\begin{equation}
  \label{flatgraph} |D f|< \eps_0(m) \qquad\mbox{on $P$ }
\end{equation}
for some number $\eps_0$ to be specified later on. The ultimate goal is to show
that $\osc_{\Ball^m(b,s)}Df\le Cs^\tau$ with a constant $C$ depending only on
the~local energy of $\Sigma$; cf.~\eqref{FINAL}. The smallness condition
\eqref{flatgraph} allows us to use all estimates of Section~\ref{sec:2} for all
tangent planes $T_z\Sigma$ with $z\in \Sigma\cap \Ball^n(0,20 R)$.

\reallyinvisible{}{ 
  The essence is to be able to apply all estimates of Section \ref{sec:2}\xx for
  all tangent planes $T_z\Sigma$ with $z\in \Sigma\cap \Ball^n(0,20\xx R)$: we
  want these planes to be sufficiently close in $G(n,m)$.  {\tt\xx does the
    preceding sentence really help???\xx}
}
Let $F\colon P\to \R^n$ be the natural parametr\-ization of $\Sigma\cap
\Ball^n(0,20 R)$, given by $F(\xi) = (\xi,f(\xi))$ for $\xi\in P$; outside
$\Ball^n(0,20 R)$ the image of $F$ does not have to coincide with $\Sigma$.
The choice of $\eps_0$  guarantees
\begin{equation}
  \ang (T_{F(\xi_1)}\Sigma, T_{F(\xi_2)}\Sigma ) <  \eps_1(m) \qquad \mbox{for all $x_1,x_2\in
    \Ball^n(0,5R)\cap P$,} \label{tan-angles}
\end{equation}
where $\eps_1(m)$ is the constant from Lemma~\ref{proj-slab}.

As in our papers \cite[Section~6]{svdm-tp1}, \cite{svdm-surfaces} and
\cite{slawek-phd}, developing the idea which has been used in
\cite{ssvdm-triple} for curves, we introduce the maximal functions controlling
the oscillation of $D f$ at various places and scales,
\begin{equation}
  \label{Phi-ast} \Phi^\ast(\varrho, A) = \sup_{{B_\varrho\subset A}}
  \left(\osc_{B_\varrho} D f\right)
\end{equation}
where the supremum is taken over all  possible closed $m$-dimensional balls
$B_\varrho$ of radius $\varrho$ that are contained in a subset  $A
\subset \Ball^n(0,5R) \cap P$, with $\varrho\le 5R$. Since $f\in
C^{1,\kappa}$ with $\kappa=\kappa_1$ or $\kappa=\kappa_2$ we have a priori
\begin{equation}
  \label{apriori} \Phi^\ast(\varrho, A)\le C\varrho^{\kappa_i}, \qquad i=1 \mbox{ or } i=2, 
\end{equation}
for some constant $C$ which does not depend on $\varrho,A$.

To show that $f\in C^{1,\tau}$ for $\tau=1-m/p$, we check that locally, on each
scale $\rho$, the oscillation of $D f$ is controlled by a main term which
involves the local integral of $\GCi(z)^p$ and has the desired form
$C\rho^\tau$, up to a small error, which itself is controlled by the oscillation
of $D f$ on a much smaller scale $\rho/N$. The number $N$ can be chosen so large
that upon iteration this error term vanishes.

\begin{lemma}\label{key}
  Let $f$, $F$, $\Sigma$, $R>0$ and $P$ be as above. If $z_1,z_2\in
  \Ball^n(0,2R)\cap P$ with $|z_1-z_2|=t>0$, then for each sufficiently large
  $N>4$ we have
  \begin{equation}
    \label{osc-error}
    |D f(z_1)-D f(z_2)| \le  A(m)\Phi^\ast (2t/N,B) + C(N, m,p) \,  E_B^{1/p} \, t^\tau
  \end{equation}
  where $B:=\Ball^m(\frac{z_1+z_2}2,t)$ is an $m$-dimensional disk in $P$,
  $\tau:=1-m/p$, and
  \begin{equation}
    \label{EonB} E_B= \int_{F(\! B)}\GCi(z)^p
    \,\, d\H^m(z)
  \end{equation}
  is the local curvature energy of $\Sigma$ (with $i=1$ or $i=2$, respectively)
  over $B$. In the case of global tangent-point curvature $\GTP$ one can use
  \eqref{osc-error} with $A(m)=2$.
\end{lemma}

\noindent\textbf{Remark.} Once this lemma is proved, one can fix
an $m$-dimensional disk $$\Ball^m(b,s)\subset \Ball^n(0,R)\cap P$$ and use
\eqref{osc-error} to obtain for $ t\le s$
\begin{multline}
  \label{pre-morrey} \Phi^\ast(t, \Ball^m(b,s))\le A(m) \Phi^\ast
  \bigl(4t/N, \Ball^m(b,s+2t) \bigr) \\{}+ C( N,m,p)\, M_p^i
  (b,s+2t)\,
  t^\tau\, , \quad \tau=1-\frac mp\, ,
\end{multline}
where
\[
M_p^i (b,r):= \left(\int_{F(\! \Ball^m(b,r))}
  \GCi(z)^p \,\, d\H^m(z)\right)^{1/p}\quad\Fo i=1,2.
\]
We fix $i$ and then  a large $N=N(i,m,p)>4$ such that  
$A(m)(4/N)^{\kappa_i}<1/2$. This yields $A(m)^j\cdot (2/N)^{j\kappa_i}\to
0$ as $j\to \infty$. Therefore, one can iterate
\eqref{pre-morrey} and eventually show that
\begin{eqnarray}\label{FINAL}
  \osc_{\Ball^m(b,s)} D f  & \le & C'( m,p) M_p^i(b, 5s)\cdot s^\tau
  \\
  & = & C'( m,p)\left(\int_{F(\! \Ball^m(b,5s))}
    \GCi(z)^p \,\, d\H^m(z)\right)^{1/p}\, \cdot s^\tau\, , \quad \tau=1-\frac mp\, .\notag
\end{eqnarray}
Thus, in particular, we have the following.

\begin{corollary}[\textbf{Geometric Morrey-Sobolev embedding into $C^{1,\tau}$}]\label{cor:optimal}
  Let $p>m$ and 
  $\Sigma \subset \R^n$ be an $m$-fine set 
  \[
  \int_\Sigma \GCi(z)^p\, d\H^m(z) \le E <\infty
  \]
  for $i=1$ or $i=2$. Then $\Sigma$ is an embedded closed manifold of class
  $C^{1,\tau}$, where $\tau = 1 - m/p$.  Moreover we can find a radius $R_2
  =a_2(n,m,p) E^{-1/(p-m)}\le R_1$, where $a_2(n,m,p)$ is a constant depending
  only on $n$, $m$ and $p$, and a constant $K_2 = K_2(n,m,p)$ such that for each
  $x \in \Sigma$ there is a function
  $$
  f:T_x\Sigma=:P
  \cong\R^m\to P^\perp\cong\R^{n-m}
  $$
  of class $C^{1,\tau}$, such that $f(0)=0$ and
  $D f(0)=0$, and
  $$
  \Sigma\cap\Ball^n(x,R_2)=x+\Big(\graph f\cap\Ball^n(0,R_2)\Big),
  $$
  where $\graph f\subset P\times P^\perp=\R^n$ denotes the
  graph of $f$, and we have
  \begin{equation}
    \label{pre-ptwise}
    |D f(z_1)-D f(z_2)|\le K_2\biggl(\int_{U(z_1,z_2)} \GCi\bigl((z,f(z)\bigr)^p\, dz\biggr)^{1/p} |z_1-z_2|^\tau
  \end{equation}
  for all $z_1,z_2\in \Ball^n(0,{R_2}) \cap P$, where
  \[
  U(z_1,z_2) = \Ball^m((z_1+z_2)/2,5|z_1-z_2|) \,. 
  \]
\end{corollary}

\medskip

The rest of this section is devoted to the proof of Lemma~\ref{key} for each
of the global curvatures $\GCi$. We follow the lines of \cite{slawek-phd} and
\cite{svdm-tp1} with some technical changes and necessary adjustments.

\subsubsection{Slicing: the setup. Bad and good points.}

We fix $z_1,z_2$ and the disk $B=\Ball^m(\tfrac{z_1+z_2}{2},t)$ as in the
statement of Lemma \ref{key}; we have $\H^m(B)=\omega_mt^m$. Pick $N>4$ and let
$E_B$ be the curvature energy of $\Sigma$ over $B$, defined for $i=1$ or $i=2$
by \eqref{EonB}. Assume that $D f\not\equiv \mathrm{const}$ on $B$, for
otherwise there is nothing to prove.

\smallskip 
Take
\begin{equation}
  \label{K0}
  K_0:= \left( E_B\cdot N^{m} \omega_m^{-1}\right)^{1/p} > 0
\end{equation}
and consider the set of \emph{bad points} where the global curvature becomes
large,
\begin{equation}
  Y_0 := \{ \xi\in B \ \colon \
  \GCi(F(\xi))>  K_0t^{-1+\tau}=K_0t^{-m/p}\}\, .\label{bad}
\end{equation}
We now estimate the curvature energy to obtain a bound for $\H^m(Y_0)$. For this
we restrict ourselves to a portion of $\Sigma $ that is described as the graph
of the function $f$.
\begin{eqnarray*}
  E_B & =&\int_{F( B)} \GCi(z)^p \,\,
  d\H^m(z)\\
  &\ge & \int_{F( Y_0)} \GCi(z)^p\,d\H^m(z)\\
  &  = & 
  \int_{Y_0}\GCi(F(\xi))^p\sqrt{\det\Big(\Big[\begin{array}{c}
      \Id_{\R^m}\\
      Df(\xi)\end{array}\Big]^T
    \Big[\begin{array}{c}
      \Id_{\R^m}\\
      Df(\xi)\end{array}\Big]\Big)}
  \,d\xi
  \\ 
  &\ge & \int_{Y_0} \GCi(F(\xi))^p\, d\xi \\
  &\stackrel{\eqref{bad}}{>} & \H^m(Y_0)
  K_0^p t^{-m} \ = \ E_B \H^m(Y_0) N^{m}
  \bigl(\H^m(B)\bigr)^{-1}\, .
\end{eqnarray*}
The last equality follows from the choice of $K_0$ in \eqref{K0}. Thus, we obtain
\begin{equation}
  \label{smallbad}
  \H^m(Y_0) < \frac{1}{N^m} \H^m(B)=\omega_m
  \frac{t^m}{N^m},
\end{equation}
and since the radius of $B$ equals $t$, we obtain
\begin{equation}
  \label{close}
  \Ball^m (z_j,t/N) \setminus Y_0 \not=\emptyset
  \qquad\mbox{for $j=1,2$.}
\end{equation}
Now, select two \emph{good points} $u_j \in \Ball^m(z_j,t/N) \setminus Y_0$
($j=1,2$). By the triangle inequality,
\begin{eqnarray}
  |D f(z_1)-D f(z_2)| & \le & |D f(z_1)-D f
  (u_1)| + |D f(u_2)-D f(z_2)|\nonumber \\
  & & {} + |D f(u_1)-D f(u_2)| \nonumber \\
  & \le & 2 \Phi^\ast(t/N,B) + |D f(u_1)-D f(u_2)|\, .\label{zj-uj}
\end{eqnarray}
Thus, we must only show that for \emph{good} $u_1,u_2$ the 
last term in \eqref{zj-uj} 
satisfies
\begin{equation}
  \label{goodpts-est}
  |D
  f(u_1)-D f(u_2)|\le 
  A(m) \Phi^\ast(2t/N,B) + C(N,m,p) E_B^{1/p} t^\tau\, .
\end{equation}
This has to be done for each of the global curvatures $\GCi$. (It will turn out
that for $\GTP$ one can use just the second term on the right hand side of
\eqref{goodpts-est}.)

\subsubsection{Angles between good planes: the `tangent-point' case}

We first deal with the case of $\GTP$ which is less complicated. To verify
\eqref{goodpts-est}, we assume that $D f(u_1)\not= D f(u_2)$ and work with the
portion of the surface parame\-tri\-zed by the points in the \emph{good set}
\begin{equation}
  \label{G} G:= B\setminus Y_0 .
\end{equation}
By \eqref{smallbad}, $G$ satisfies
\begin{equation}
  \label{G-below} \H^m(G) > (1- N^{-m}) \H^m(B) =: C_1( p,m)\,
  t^m\, .
\end{equation}
To conclude the whole proof, we shall derive -- for each of the two global
curvatures -- an upper estimate for the measure of $G$,
\begin{equation}
  \label{G-above} \H^m(G) \le  C_2(p,m)\, K_0\,
  \frac{t^{m+\tau}}{\alpha}, \,
\end{equation}
where $\alpha:=\ang (H_1,H_2)\not= 0$ and $H_i:=T_{F(u_i)}\Sigma$
denotes the tangent plane to $\Sigma$ at $F(u_i)\in \Sigma$ for $i=1,2.$
Combining \eqref{G-above} and \eqref{G-below}, we will then obtain
\[
\alpha <   (C_1)^{-1} C_2 K_0 t^\tau =:  C_3 E_B^{1/p} 
t^\tau\,
.
\]
(By an elementary reasoning analogous to the proof of Theorem~5.7
in \cite{svdm-tp1}, this also yields an estimate for the oscillation of $D f$.)

Following \cite[Section~6]{svdm-tp1} closely, we are going to 
prove the upper estimate \eqref{G-above} for $\H^m(G)$. 

By Corollary \ref{C1kappa-second} and Remark \ref{rem:flat}
$$
\Sigma\cap\Ball^n(F(u_1),20R)=F(u_1)+\Big(\graph
f_1\cap\Ball^n(0,20R)\Big),
$$
i.e, that portion of $\Sigma$ near $F(u_1)\in\Sigma$ is a graph of a
$C^{1,\kappa_2}$ function $f_1:H_1:= T_{F(u_1)}\Sigma\to H_1^\perp$ with
$|\nabla f_1|<\eps_0(m)\ll 1.$ As $G\subset B=\Ball^m(\tfrac{z_1+z_2}{2},t)$
with $z_i\in \Ball^n(0,2R)\cap P,$ $t=|z_1-z_2|\le 4R,$ and
$u_i\in\Ball^m(z_i,t/N)$ (see \eqref{close}), we have the inclusion
$$
G\subset \Ball^m(0,6R)\subset\Ball^m(u_1,6R+2R+t/N)\subset\Ball^m(u_1,10R),
$$
and, as $F$ is $2$-Lipschitz, $F(G)\subset\Ball^n(F(u_1),20R),$
i.e., $F(G)\subset x+\Big(\graph f_1\cap\Ball^n(0,20R)\Big).$
Thus, since $\eps_0(m)$ is small,
\begin{eqnarray*}
  \H^m(F(G))&= &
  \int_{\pi_{H_1}(F(G))}\sqrt{\det\Big(\Big[\begin{array}{c}
      \Id_{\R^m}\\
      Df_{1}(\xi)\end{array}\Big]^T
    \Big[\begin{array}{c}
      \Id_{\R^m}\\
      Df_{1}(\xi)\end{array}\Big]\Big)}
  \,d\xi\\
  &< &\int_{\pi_{H_1}(F(G))}\sqrt{2}\,d\xi=\sqrt{2}\H^m(\pi_{H_1}(F(G))).
\end{eqnarray*}
Therefore, 
$$
\H^m(G)\le\H^m(F(G))<\sqrt{2}\H^m(\pi_{H_1}(F(G))),
$$
so that \eqref{G-above} would follow from
\begin{equation}
  \H^m\bigl(\pi_{H_1}(F(G))\bigr)\le C_4(m)\, K_0\,
  \frac{t^{m+\tau}}{\alpha}\, . \label{proj-G-H1}
\end{equation}
To achieve this, we shall use the definition of $\GTP$ combined with the
properties of intersections of tubes stated in Lemma~\ref{proj-slab}. To shorten
the notation, we write
\[
\frac{1}{\rtp(x,y;T_x\Sigma)}\equiv\frac{1}{\rtp(x,y)}\, , \qquad x,y\in \Sigma\, .
\]
For an arbitrary $\zeta \in G$ and $i=1,2$ we have by \eqref{bad}
\begin{align*}
  \frac{1}{\rtp(F(u_i),F(\zeta))}& =
  \frac{2\bigl|Q_{H_i}(F(\zeta)-F(u_i))\bigr|}{|F(\zeta)-F(u_i)|^2}
  \\
  &\le \GTP(F(u_i))
  \le
  K_0 t^{-1+\tau}\, .
\end{align*}
Let $P_i=F(u_i)+H_i$ be the affine tangent plane to $\Sigma$ at
$F(u_i)$. Since $F$ is Lipschitz with constant $(1+\eps_0)<2$
and
$|\zeta-u_i|\le  2 t$,
\begin{eqnarray}
  \dist (F(\zeta),P_i) & = & \dist (F(\zeta)-F(u_i), H_i) \label{z-Pi} \\
  & = & \bigl|Q_{H_i}(F(\zeta)-F(u_i))\bigr|\  <  8K_0
  t^{1+\tau} =: h_0 \nonumber
\end{eqnarray}
for $\zeta\in G$, $i=1,2$.  Select the points $p_i\in P_i$, $i=1,2$,
so that $|p_1-p_2|=\dist (P_1,P_2)$. The vector $p_2-p_1$ is then
orthogonal to $H_1$ and to $H_2$, and since $G$ is nonempty by
\eqref{G-below}, we have $|p_1-p_2| < 2h_0$ by \eqref{z-Pi}.

Set $p=(p_1+p_2)/2$, pick a parameter $\zeta\in G$ and consider
$y=F(\zeta)-p$. We have
\[
y= (F(\zeta) - F(u_1)) + (F(u_1)-p_1) + (p_1-p),
\]
so that $\pi_{H_1}(y) = \pi_{H_1} (F(\zeta)-F(u_1)) + (F(u_1) - p_1)
$, and
\begin{eqnarray*}
  |y-\pi_{H_1}(y)| & = & |(p_1-p) + F(\zeta) - F(u_1) -  \pi_{H_1}
  (F(\zeta)-F(u_1))|\\
  & = & |(p_1-p) + Q_{H_1}(F(\zeta) - F(u_1)) |
  \, .
\end{eqnarray*}
Therefore, since $|p-p_1|\le h_0$ and by  \eqref{z-Pi}, $
|y-\pi_{H_1}(y)| < h_0 + h_0 = 2h_0$. In the same way, we obtain
$|y-\pi_{H_2}(y)| < 2h_0$. Thus,
\[
\frac{y}{2h_0}=\frac{F(\zeta)-p}{2h_0} \in S(H_1,H_2),
\]
where $S(H_1,H_2)=\{x\in \R^n \colon \dist(x,H_j)\le 1\mbox{ for
  $j=1,2$}\}$ is the intersection of two tubes around the planes $H_j$ considered in
Section~2.2. Applying Lemma~\ref{proj-slab} which is possible due to the estimate \eqref{tan-angles} for $\ang(H_1,H_2)$, we
conclude that there exists an $(m-1)$-dimensional subspace
$W\subset H_1$ such that
\begin{equation}
  \label{strip} \pi_{H_1}(F(G)-p) \subset \{x\in H_1\, \colon\,
  \dist (x,W)\le 2h_0\cdot 5c_2/\alpha\}\, .
\end{equation}
On the other hand, since $F$ is $2$-Lipschitz, 
we certainly have
$$F(G)\subset  \Ball^n\Bigl(F(\frac{z_1+z_2}2),2t\Bigr)$$ 
and therefore
\begin{equation}
  \label{ball} \pi_{H_1}(F(G)-p)\subset \Ball^n(a,2t), \qquad
  a:=\pi_{H_1}(F( \frac{z_1+z_2}2)-p).
\end{equation}
Combining \eqref{strip}--\eqref{ball}, we use
Lemma~\ref{strip-ball} for the plane $H:=H_1\in G(n,m)$, the set $S':=\pi_{H_1}(F(G)-p)$, and
$d:=2h_05c_2/\alpha$, to obtain
\begin{equation}
  \label{last-gtp}
  \H^m\bigl(\pi_{H_1}(F(G))\bigr) \le 4^{m-1}t^{m-1} \cdot 20h_0
  c_2/\alpha =: C_4 (m) K_0\frac{t^{m+\tau}}{\alpha}
\end{equation}
by definition of $h_0$ in \eqref{z-Pi}, which is the desired \eqref{proj-G-H1}, implying \eqref{G-above} and thus
completing the bootstrap estimates in the case of the global tangent-point curvature $\GTP$.

\subsubsection{Angles between good planes: the `Menger' case}

To obtain \eqref{goodpts-est} for the global Menger curvature $\GMC$, one
proceeds along the lines of \cite{slawek-phd}, with a few necessary changes.

The main difference between $\GTP$ and $\GMC$ is that the control of $\GTP$
directly translates to the control of the angles between the tangent
planes. In the case of $\GMC$ an extra term is necessary. Namely, we choose
$x_1,\ldots, x_m\in P$ so that
\[
|x_i-u_1| =\frac tN, \qquad i=1,2,\ldots,m
\]
and the vectors $x_i-u_1$ form and ortho-$\rho$-normal basis of $P$ with
$\rho=t/N$; see Definition \ref{ortho_basis}.  Analogously, we choose
$y_1,\ldots, y_m\in P$ close to $u_2$. Next, setting as before
$H_j=T_{F(u_j)}\Sigma$, we write
\begin{align}
  |D f(u_1)-D f(u_2)|&\lesssim \ang (H_1,H_2)\label{tan-sec}\\
  &\le \ang (H_1,X) + \ang(X,Y) + \ang (Y,H_2)\, ,\notag
\end{align}
with the constant in \eqref{tan-sec} depending on $m$ only, where
\begin{align*}
  X &= \mathrm{span}\, (F(x_1)-F(u_1),F(x_2)-F(u_1),\ldots,F(x_m)-F(u_1))\\
  Y &= \mathrm{span}\, (F(y_1)-F(u_2),F(y_2)-F(u_2),\ldots,F(y_m)-F(u_2))\, 
\end{align*}
are the secant $m$-dimensional planes, approximating the tangent ones. A
technical but routine calculation, relying on the fundamental theorem of
calculus (see e.g.  \cite[Proof of Thm. 4.3]{slawek-phd} or (for $m=2$) Step 4
of the proof of Theorem 6.1 in \cite{svdm-surfaces}), shows that if the
constant $\eps_0=\eps_0(m)>0$ controlling the oscillation of $Df$ is chosen
small enough then
\[
\ang (H_1,X)+\ang(Y,H_2)\le C(m) \Phi^\ast (2t/N,B)\, ,
\]
and consequently
\begin{equation}
  \label{almost-end}
  |D f(u_1)-D f(u_2)|\le A(m) \Phi^\ast (2t/N,B) + C(m) \ang(X,Y) \,,
\end{equation}
where $C(m)$ comes from~\eqref{tan-sec}. Thus, it remains to estimate the
angle between the secant planes $X,Y$ approximating the tangent ones
$H_1,H_2$. The estimate of $\ang(X,Y)$ is very similar to the computations
carried out in Section 3.3.2 for the global-tangent point curvature. Here is the
crux of the argument.

We let $G=B\setminus Y_0$ be the good set defined in \eqref{G}. Shrinking
$\eps_0=\eps_0(m)$ if necessary, we may assume that
\begin{equation}\label{SMALLANGLE}
  \ang (X,Y)\le \eps_1(m)
\end{equation}
where $\eps_1(m)$ is sufficiently small. Then,
\[
\H^m(G)\le \H^m(F(G))\le 2H^m(\pi_X(F(G)))\, ,
\]
and the strategy is to show a counterpart of \eqref{proj-G-H1}, namely
\begin{equation}
  \label{proj-G-X}
  \H^m\bigl(\pi_{X}(F(G))\bigr)\le C_5\, K_0\,
  \frac{t^{m+\tau}}{\alpha}\, ,\qquad\alpha=\ang(X,Y)\, .
\end{equation}
Comparing this estimate with the lower bound \eqref{G-below} for the measure of
$G$, one obtains
\[
\ang(X,Y)\lesssim K_0t^\tau= \mathrm{const}\cdot E_B^{1/p} t^\tau
\]
which is enough to conclude the proof of Lemma~\ref{key} also in the case of the
global Menger curvature $\GMC$.

Now, to verify \eqref{proj-G-X}, we select a point 
$\zeta\in B=\Ball^m(\tfrac{z_1+z_2}{2},t)$ with
\[
|\zeta-u_j|\approx |F(\zeta)-F(u_j)|\approx \frac t2, \qquad 
j=1,2
\]
(one can arrange to have constants here close to $1$ by the initial uniform
smallness of $\varepsilon_0(m)$ in \eqref{flatgraph}). Then, the $(m+1)$-simplex
$T$ with vertices at $F(u_1)$, $F(x_1)$, \ldots, $F(x_m)$, $F(\zeta)$ is of
diameter $\approx t$. The face
\[
\face_{m+1}(T)=\conv\bigl\{F(u_1), F(x_1), \ldots, F(x_m)\bigr\}
\]
is spanned by $m$ nearly orthogonal edges $F(x_i)-F(u_1)$, of length roughly
$t/N$ each, and therefore $\H^m(\face_{m+1}(T))\approx t^m$. Thus, setting now
$P_1=F(u_1)+X$, and keeping in mind that $u_1\not \in Y_0$ (see \eqref{bad}), we
obtain by means of \eqref{volume_simplex}
\begin{align*}
  K_0t^{-1+\tau} &\ge \GMC(F(u_1))\\
  &\ge K(F(u_1),F(x_1), \ldots,F(x_m),F(\zeta))
  \approx \frac{t^m\dist(F(\zeta),P_1)}{t^{m+2}}.
\end{align*}
Thus,
\begin{equation}
  \label{zeta-P1} \dist(F(\zeta),P_1)\le C(m) K_0 t^{1+\tau}\, ,
\end{equation}
and the same estimate holds for $\dist(F(\zeta),P_2)$ where
$P_2=F(u_2)+Y$. Thus, we have a counterpart of \eqref{z-Pi} in the previous
subsection. From that point we reason precisely like in Section 3.3.2, between
\eqref{z-Pi} and \eqref{last-gtp}, where at one point we need to use
\eqref{SMALLANGLE}.  This completes the proof of Lemma~\ref{key} in the case of
global Menger curvature $\GMC$.

\subsection{$W^{2,p}$ estimates for the graph patches}

We now show that Corollary~\ref{cor:optimal} combined with the result of Haj\l{}asz, cf. Theorem~\ref{hajlasz-ptwise}, easily yields the following.

\begin{theorem}[\textbf{Sobolev estimates}] \label{W2p-estimates}
  Let $\Sigma \subset \R^n$ be an $m$-fine set with
  \[
  \int_\Sigma \GCi(z)^p\, d\H^m(z) \le E <\infty
  \]
  for $i=1$ or $i=2$. Then $\Sigma$ is an embedded closed manifold of class $C^{1,\tau}\cap W^{2,p}$, where $\tau = 1 - m/p$.

  Moreover we can find a radius $R_3 = a_3(n,n,p)E^{-1/(p-m)}\le
  R_2$,where $a_3(n,m,p)$ is a constant depending only on 
  $n,m$, and $p$, 
  and a constant $K_3 =
  K_3(n,m,p)$ such that
  for each $x \in \Sigma$  there is a function
  $$
  f:T_x\Sigma=:P
  \cong\R^m\to P^\perp\cong\R^{n-m}
  $$
  of class $C^{1,\tau}\cap W^{2,p}$, such that $f(0)=0$
  and $Df(0)=0$, and 
  $$
  \Sigma\cap\Ball^n(x,R_3)=x+\Big(\graph f\cap\Ball^n(0,R_3)\Big),
  $$
  where $\graph f\subset P\times P^\perp=\R^n$ denotes the
  graph of $f$
\end{theorem}

\begin{proof}
  It remains to show that the graph parametrizations are in fact in
  $W^{2,p}$. To this end, we fix an exponent $s\in (m,p)$ and apply
  Corollary~\ref{cor:optimal} with $p$ replaced by $s$, to obtain from
  \eqref{pre-ptwise} the following estimate
  \begin{eqnarray*}
    \lefteqn{
      |D f(z_1)-D f(z_2)|}\\
    & \lesssim & \biggl(\int_{\Ball^m((z_1+z_2)/2,5|z_1-z_2|)} 
    \GCi\bigl((z,f(z))\bigr)^s\, dz\biggr)^{1/s} |z_1-z_2|^{1-m/s}\\
    & \lesssim & \biggl(\fint_{\Ball^m((z_1+z_2)/2,5|z_1-z_2|)} 
    \GCi\bigl((z,f(z))\bigr)^s\, dz\biggr)^{1/s} |z_1-z_2|\\
    & \lesssim & \bigl(G(z_1) + G(z_2)\bigr)
    |z_1-z_2|
  \end{eqnarray*}
  where
  \[
  G(z)=\Bigl(M \GCi\bigl(F(z)\bigr)^s\Bigr)^{1/s} \quad\Fo F(z)=(z,f(z))\, ,
  \]
  and $M h$ denotes the standard Hardy-Littlewood maximal function of
  $h$. Since $p>s$, we have $p/s> 1$, so that $(\GCi\circ F)^s$ is in
  $L^{p/s}$ and by the Hardy--Littlewood maximal theorem $\bigl.G\, \bigr.^s=M
  \bigl((\GCi\circ F)^s\bigr)\in L^{p/s}$.  Thus, $G\in L^p$. An application
  of Haj\l{}asz' Theorem~\ref{hajlasz-ptwise} concludes the proof of
  Theorem~\ref{W2p-estimates}.
\end{proof}

\section{From $W^{2,p}$ estimates to finiteness of both energies}

\label{sec:4}
\setnumbers

In this section, we prove the implications (1) $\Rightarrow$ (2), (3) of the
main result, Theorem \ref{mainthm}. Let us begin with a definition.

\begin{definition}
  \label{def:w2pmani}
  Let $\Sigma \subset \R^n$. We say that $\Sigma$ is an \emph{$m$-dimensional,
    $W^{2,p}$-manifold} (without boundary) if at each point $x \in \Sigma$ there
  exist an $m$-plane $T_x\Sigma \in G(n,m)$, a radius $R_x > 0$, and a function
  $f \in W^{2,p}( T_x\Sigma \cap \Ball^n(0,2R_{x}),\R^{n-m})$ such that
  \begin{displaymath}
    \Sigma \cap \Ball^n(x,R_x) = 
    x+\Big( \graph f \cap\Ball^n(0,R_x) \Big)\,.
  \end{displaymath}
\end{definition} 

We will use this definition only for $p>m$. In this range, by the Sobolev
imbedding theorem,  each~$W^{2,p}$-manifold is a manifold of class $C^1$.

\begin{theorem}
  \label{thm:w2p-gmc}
  Let $p > m$ and let $\Sigma$ be a compact, $m$-dimensional,
  $W^{2,p}$-manifold. Then the global curvature functions $\GMC[\Sigma]$ and
  $\GTP[\Sigma]$ are of class $L^p(\Sigma,\H^m)$.
\end{theorem}

\begin{remark}
  As already explained in the introduction, here we assume that $\GTP$ is
  defined for the natural choice of $m$-planes $H_x=T_x\Sigma$.  As we
  mentioned before, if $\Sigma$ is a $C^1$ manifold and $H_x \ne T_x\Sigma$ on a
  set of positive $\H^m$-measure, then the global curvature $\GTP$ defined for
  $H_x$ instead of $T_x\Sigma$ has infinite $L^p$-norm.
\end{remark}

\subsection{Beta numbers for $W^{2,p}$ graphs}

We start the proof with a general lemma that shall be applied later to obtain specific estimates for $\GMC$ and $\GTP$
in $L^p(\Sigma)$.

\begin{lemma}
  \label{lem:w2p-beta-est}
  Let $f \in W^{2,p}(\Ball^m(0,2R),\R^{n-m})$, where $p > m$ and let $\Sigma =
  \graph f $. Then there exists a function $g \in L^p(\Sigma
  \cap\Ball^n((0,f(0),2R),\H^m)$ such that
  for each $a \in \Sigma \cap \Ball^n((0,f(0)),R)$
  and any $r < R$
  \begin{displaymath}
    \beta_{\Sigma}(a,r) \le g(a) r \,.
  \end{displaymath}
\end{lemma}

\begin{proof}
  Fix $s \in (m,p)$.  Then, $f \in W^{2,s}(\Ball^m(0,2R))$. Since $s > m$ we
  have the embedding $$W^{2,s}(\Ball^m(0,2R)) \subset
  C^{1,\alpha}(\overline{\Ball^m(0,2R)}),$$ where $\alpha = 1 - \frac
  ms$. Choose some point $x \in \Ball^m(0,R)$ and set as before
  \begin{displaymath}
    F(z):=(z,f(z)) 
    \quad \text{and} \quad
    \Psi_x(z):=F(z)-DF(x)(z-x)\quad\Fo z\in \Ball^m(0,2R).
  \end{displaymath}
  
  Of course $\Psi_x$ is in $W^{2,p}(\Ball^m(0,2R),\R^n)$ and
  therefore also in $W^{2,s}(\Ball^m(0,2R),\R^n)$. We now 
  fix another point $y$ in $\Ball^m(x,R)$ and estimate the 
  oscillation of $\Psi_x$. Set
  \[
  U:=\Ball^m\Bigl(\frac{x+y}2,|x-y|\Bigr)
  \]
  By two consecutive applications of the Sobolev imbedding
  theorem in the supercritical case (cf.
  \cite[Theorem~7.17]{gilb-trud}),
  keeping in mind that $U$ is a ball of radius $|x-y|$, we obtain
  \begin{align*}
    |\Psi_x(y) - \Psi_x(x)| 
    &\le C(n,m,s) |y-x|^{1 - \frac ms} \left( \int_{U} |D\Psi_x(z)|^s\ dz \right)^{1/s} \\
    &= C' |y-x| \left( \fint_{U} |D\Psi_x(z)|^s\ dz \right)^{1/s} \\
    &= C' |y-x| \left( \fint_{U} |D F(z) - D F(x)|^s\ dz 
    \right)^{1/s} \\
    &\le \tilde{C} |y-x| \left(
      \fint_{U}
      |z-x|^{s-m} \int_{U} |D^2F(w)|^s\ dw
      \ dz
    \right)^{1/s} \\
    &= \bar{C} |y-x|^2 \left( \fint_{\Ball^m(\frac{x+y}2,|x-y|)} |D^2f(w)|^s\ dw \right)^{1/s}\\ 
    &\le \hat{C} |y-x|^2 M
    (|D^2f|^s)^{1/s}(x) \,.
  \end{align*}
  Here $M$ denotes the Hardy-Littlewood maximal function and the constant
  $\hat{C} = \hat{C}(n,m,s)$ depends on $n,m,$ and $s$.  Since $m < s < p$ we
  have $\frac ps > 1$ and $|D^2f|^s \in L^{p/s}(\Ball^m(0,2R))$. Hence we also
  have $M(|D^2f|^s) \in L^{p/s}(\Ball^m(0,2R))$. Therefore $M(|D^2f|^s)^{1/s}
  \in L^p(\Ball^m(0,2R))$.

  To estimate the $\beta$ number, note that
  \[
  |\Psi_x(y) - \Psi_x(x)| = |F(y) - F(x) - DF(x)(y-x)| = 
  |f(y) - f(x) - Df(x)(y-x)| \,.
  \]
  Choose two points $a \in \Sigma \cap \Ball^n(F(0),R)$ and $b \in \Sigma \cap
  \Ball^n(F(0),2R)$. Since $\Sigma = \graph f $ there exist $x,y \in
  \Ball^m(0,2R)$ such that $F(x) = a$ and $F(y) = b$.

  Of course we have $|y-x| \le |b-a|$. Now we obtain
  \begin{align*}
    \dist(b, a + T_a\Sigma) &= \dist(F(y), F(x) + T_{F(x)}\Sigma) \\
    &\le |F(y) - F(x) - DF(x)(y-x)| \\
    &= |\Psi_x(y) - \Psi_x(x)|\\
    &\le \hat{C} |y-x|^2 M(|D^2f|^s)^{1/s}(x) \\
    &\le \hat{C} |b-a|^2 M(|D^2f|^s)^{1/s}(\pi_{\R^m}
    (a)) \,.
  \end{align*}
  Since $\pi_{\R^m}$ is bounded we find together with the previous
  considerations that the function $g(a) := \hat{C}
  M(|D^2f|^s)^{1/s}(\pi_{\R^m}(a))$ is of class
  $L^p(\Sigma\cap\Ball^n(F(0),2R),\H^m)$. Choose a radius $r \in (0,R]$. We have
  \begin{displaymath}
    \sup_{b \in \Sigma \cap \Ball^n(a,r)} \dist(b, a + T_a\Sigma)
    \le \sup_{b \in \Sigma \cap \Ball^n(a,r)} |b-a|^2 g(a)
    \le r^2g(a)\,.
  \end{displaymath}
  Hence
$$
    \beta_\Sigma(a,r) = \frac 1r \inf_{H \in G(n,m)} \bigg( 
    \sup_{b \in \Sigma \cap \Ball^n(a,r)} \dist(b, a + H) \bigg)
    \le \frac 1r \sup_{b \in \Sigma \cap \Ball^n(a,r)} 
    \dist(b, a + T_a\Sigma)
    \le g(a) r \,.
    $$
\end{proof}

We now need to estimate the global curvatures in terms of $\beta$
numbers. Combining these estimates with the previous lemma, we will later be
able to conclude the proof of Theorem~\ref{thm:w2p-gmc}.

\subsection{Global Menger curvature for $W^{2,p}$ graphs}

Let us begin with an estimate for the global Menger curvature $\GMC$.

\begin{lemma}
  \label{lem:dc-beta}
  Let $\Sigma \subset \R^n$ be a closed $m$-dimensional set. Choose $m+2$ points
  $x_0$,\ldots,$x_{m+1}$ of $\Sigma$; set $T = \conv(x_0,\ldots,x_{m+1})$ and $d
  = \diam(T)$. There exists a constant $C = C(m,n)$ such that
  \begin{displaymath}
    \H^{m+1}(T) \le C \beta_\Sigma(x_0,d) d^{m+1} 
  \end{displaymath}
  and
  \begin{displaymath}
    \DC(x_0,\ldots,x_{m+1}) \le C \frac{\beta_\Sigma(x_0,d)}{d} \,.
  \end{displaymath}
\end{lemma}

\begin{proof}
  If the affine space $\aff\{x_0, \ldots, x_{m+1}\}$ is not $(m+1)$-dimensional
  then $\H^{m+1}(T) = 0$ and there is nothing to prove. Hence, we can assume
  that $T$ is an $(m+1)$-dimensional simplex. The measure $\H^{m+1}(T)$ can be
  expressed by the formula (cf. \eqref{volume_simplex})
  \begin{displaymath}
    \H^{m+1}(T) 
    = \frac{1}{m+1} \dist(x_{m+1}, \aff\{x_0,\ldots,x_m\}) \H^m(\conv(x_0,\ldots,x_m)) \,.
  \end{displaymath}
  In the same way, one can express the measure $\H^m(\conv(x_0,\ldots,x_m))$
  etc.; by induction,
  \begin{displaymath}
    \H^{m+1}(T) \le \frac 1{(m+1)!} d^{m+1} \,.
  \end{displaymath}
  Hence, if $\beta_\Sigma(x_0,d) = 1$, then there is nothing to prove, so we can
  assume that $\beta_\Sigma(x_0,d) < 1$.

  Fix an $m$-plane $H \in G(n,m)$ such that
  \begin{equation}
    \label{eq:points-dist}
    \dist(y, x_0 + H) \le d \beta_\Sigma(x_0,d) 
    \qquad\text{for all } y \in \Sigma \cap \Ball^n(x_0,d)\,.
  \end{equation}
  Set $h := d \beta_\Sigma(x_0,d) < d$. Without loss of generality we can assume
  that $x_0$ lies at the origin. Let us choose an orthonormal basis
  $(v_1,\ldots,v_n)$ of $\R^n$ as coordinate system, such that
  $\Span\{v_1,\ldots,v_m\} = H$. Because of~\eqref{eq:points-dist} in our
  coordinate system we have
  \begin{displaymath}
    T \subset [-d, d]^m \times [-h,h]^{n-m}\,.
  \end{displaymath}
  Of course, $T$ lies in some $(m+1)$-dimensional section of the above
  product. Let
  \begin{align*}
    V &:= \aff \{ x_0, \ldots, x_{m+1} \} =  
    \lin \{ x_1, \ldots, x_{m+1} \} \,,\\
    Q(a,b) &:= [-a,a]^m \times [-b,b]^{n-m} \,,\\
    Q &:= Q(d,h) \\
    \text{and} \quad
    P &:= V \cap Q \,.
  \end{align*}
  Note that each of the sets $V$, $Q$ and $P$ contains $T$. Choose another
  orthonormal basis $w_1$, \ldots, $w_n$ of $\R^n$ such that $V = \lin\{w_1,
  \ldots, w_{m+1}\}$. Set
  \[
  S := \{ x \in V^{\perp} : |\langle x, w_i \rangle |\le h
  \quad\Fo i=1,\ldots,m
   \}\, .
  \]
  Thus, $S$ is just the cube $[-h,h]^{n-m-1}$ placed in the orthogonal
  complement of $V$. Note that $\diam S = 2 h \sqrt{n-m-1}$. In this setting we
  have
  \begin{equation}
    \label{eq:PxS}
    P \times S  =
    \subset Q(d + 2h \sqrt{n-m-1}, h + 2h\sqrt{n-m-1}) \,.
  \end{equation}
  Recall that $h = d \beta_\Sigma(x_0,d) < d$. We estimate
  \begin{align*}
    \H^n(T \times S) 
    &\le \H^n(P \times S) \\[4pt]
    &\le \H^n\bigl(Q(d + 2h \sqrt{n-m-1}, h + 2 h \sqrt{n-m-1}) \bigr)  \\[4pt]
    &= \bigl(2 d + 4 h \sqrt{n-m-1}\bigr)^m \bigl(2 h + 4 h \sqrt{n-m-1}\bigr)^{n-m}  \\[4pt]
    &< (2 d + 4 d  \sqrt{n-m-1})^m (2 h + 4 h \sqrt{n-m-1})^{n-m}  \\[4pt]
    &= (2 + 4 \sqrt{n-m-1})^n d^n \beta_\Sigma(x_0,d)^{n-m}  \,.
  \end{align*}
  On the other hand we have
  \begin{align*}
    \H^n(T \times S) &= \H^{m+1}(T) \H^{n-m-1}(S) \\
    &= \H^{m+1}(T) 2^{n-m-1} h^{n-m-1} \\
    &= 2^{n-m-1} \H^{m+1}(T) d^{n-m-1} 
    \beta_\Sigma(x_0,d)^{n-m-1} \,.
  \end{align*}
  Hence
  \[
  2^{n-m-1} \H^{m+1}(T) d^{n-m-1} 
  \beta_\Sigma(x_0,d)^{n-m-1}
  \le (2 + 4 \sqrt{n-m-1})^n d^n \beta_\Sigma(x_0,d)^{n-m},
  \]
  or equivalently
  \[
  \H^{m+1}(T) \le (2 + 4 \sqrt{n-m-1})^n 2^{-(n-m-1)} d^{m+1} \beta_\Sigma(x_0,d) \,.
  \]
  We may set $C = C(n,m) = (2 + 4 \sqrt{n-m-1})^n 2^{-(n-m-1)}$. This completes
  the proof of the lemma.
\end{proof}

Since $\Sigma $ is a compact $W^{2,p}$-manifold ($p>m$) we may cover it by
finitely many balls, in which $\Sigma$ is described as a graph, such that Lemma
\ref{lem:w2p-beta-est} is satisfied in each of these graph patches with a
respective function $g$ defined only on that patch.  More precisely, we find
$a_1,\ldots,a_N\in\Sigma$ with
$$
\Sigma\subset\bigcup_{k=1}^N\Ball^n(a_k,R/2),
$$
such that for each $k=1,\ldots,N,$ one has
$$
\Sigma\cap\Ball^n(a_k,2R)=a_k+\left(\graph f_k\cap\Ball^n(
  a_k,2R)\right),
$$
where $f_k\in W^{2,p}(\Ball^m(0,2R),\R^{n-m})$, and there is a function $g_k\in
L^p(\Sigma\cap\Ball^n(a_k,2R), \H^m)$ with the property that for each $a\in
\Sigma\cap\Ball^n(a_k,R)$ and any $r<R$ one has the estimate
\begin{equation}\label{beta-on-patch}
  \beta_\Sigma(a,r)\le g_k(a)r.
\end{equation}
Using a partition of unity subordinate to this finite covering, i.e.,
$(\eta_k)_{k=1}^N\subset C_0^\infty( \Ball^n(a_k,R/2))$ with $0\le\eta_k\le 1$,
$\sum_{k=1}^N \eta=1$, we can extend the functions $\eta_kg_k$ to all of
$\Sigma$ by the value zero outside of $\Ball^n(a_k,R/2)$ for each
$k=1,\ldots,N,$ and define finally $g\in L^p(\Sigma,\H^m)$ as
$$
g=\sum_{k=1}^N\eta_kg_k.
$$
Now, for any $x_0\in\Sigma$ there exists $k\in\{1,\ldots,N\}$ such that $x_0\in
\Sigma\cap\Ball^n(a_k,R/2)$, so that $\Ball^n(x_0,R/2)\subset\Ball^n(a_k,R)$,
and we conclude with \eqref{beta-on-patch} for any $r<R$
$$
\beta_\Sigma(x_0,r)=\sum_{k=1}^N\eta_k\beta_\Sigma(x_0,r)\le
\sum_{k=1}^N\eta_kg_k(x_0)r=g(x_0)r.
$$
Consequently, by Lemma \ref{lem:dc-beta},
\begin{align*}
  \GMC(x_0) &= \sup_{x_1,\ldots, x_{m+1} \in \Sigma} 
  \DC(x_0,x_1,\ldots, x_{m+1}) \\
  &\le C \sup_{x_1,\ldots, x_{m+1} \in \Sigma} 
  \frac{\beta_\Sigma(x_0,\diam(x_0,\ldots,x_{m+1}))}{\diam(x_0,
    \ldots,x_{m+1})} \\
  &\le C \sup_{x_1,\ldots, x_{m+1} \in \Sigma} g(x_0) = 
  C g(x_0) \,.
\end{align*}
This leads to the following result.
\begin{corollary}
  Let $\Sigma$ be a compact, $m$-dimensional, $W^{2,p}$-manifold for some $p >
  m$. Then $\GMC[\Sigma]\in L^p(\Sigma,\H^m).$
\end{corollary}

\subsection{Global tangent--point curvature for $W^{2,p}$ graphs}

The following simple lemma can be easily obtained from the definition of $\GTP$.

\begin{lemma}
  Assume that $\Sigma$ is a $C^1$ embedded, compact $m$-dimensional manifold
  without boundary. Then, for some $R=R(\Sigma)>0$ we have
  \[
  \GTP(x)\lesssim \frac 1R + \sup_{r<R} \frac{
    \beta_\Sigma(x,r)}{r} .
  \]
\end{lemma}

\begin{proof}
  Choose $R>0$ so that for each point $x\in \Sigma$ the intersection $\Sigma\cap
  \Ball^n(x,3R)$ is a graph of a $C^1$ function $f\colon T_x\Sigma\to
  (T_x\Sigma)^\perp$ with oscillation of $D f$ being small. Fix $x\in
  \Sigma$. Set $F(z):=(z,f(z))$ for $z\in P=T_x\Sigma$.  As before, we write
  \[
  \frac{1}{\rtp(x,y;T_x\Sigma)}\equiv\frac{1}{\rtp(x,y)}\, , \qquad x,y\in \Sigma\, .
  \]
  It is clear that for $|x-y|\ge R$ we have $\rtp(x,y)\ge R/2$ by definition. Thus
  \[
  \GTP(x)\le  \frac 2R + \sup_{|x-y|<R} \frac{1}{\rtp(x,y)}\, .
  \]
  It remains to estimate the last term. Now, if 
  $x=F(\xi)$ and $y=F(\eta)\in \Sigma\cap \Ball^n(x,R)$
  with
  \[
  |y-x|=|F(\eta)-F(\xi)|\approx |\eta-\xi|\approx \rho_j\equiv\frac{R}{2^j}\, , \qquad j=0,1,,2,\ldots,
  \]
  then
  \[
  \frac{1}{\rtp(x,y)}=\frac{2\dist(y,x+T_{x_1}\Sigma)}{|y-x|^2}
  \lesssim \frac{\beta_\Sigma(x,\rho_j)}{\rho_j}
  \]
  with an absolute constant. The lemma follows.
\end{proof}

Combining the above lemma with Lemma~\ref{lem:w2p-beta-est}, we conclude
immediately that $\GTP\in L^p$ for $W^{2,p}$-manifolds with $p>m$. The proof of
the implications (1) $\Rightarrow$ (2), (3) of Theorem~\ref{mainthm} is now
complete.

\addcontentsline{toc}{section}{References}

\small
\vspace{0cm}

\bibliography{menger}{}
\bibliographystyle{hamsplain}


\vspace*{.5cm}

\noindent
{\sc S\l{}awomir Kolasi\'{n}ski}\\
Instytut Matematyki\\
Uniwersytet Warszawski\\
ul. Banacha 2\\
PL-02-097 Warsaw \\
POLAND\\
E-mail: {\tt skola@mimuw.edu.pl}

\vspace{.5cm}

\noindent
{\sc Pawe\l{} Strzelecki}\\
Instytut Matematyki\\
Uniwersytet Warszawski\\
ul. Banacha 2\\
PL-02-097 Warsaw \\
POLAND\\
E-mail: {\tt pawelst@mimuw.edu.pl}

\vspace{.5cm}

\noindent
{\sc Heiko von der Mosel}\\
Institut f\"ur Mathematik\\
RWTH Aachen University\\
Templergraben 55\\
D-52062 Aachen\\
GERMANY\\
Email: {\tt heiko@instmath.rwth-aachen.de}

\end{document}